\documentclass[a4paper]{article}
\usepackage{amsmath}
\usepackage{amsfonts,amssymb}
\usepackage{amsfonts,amssymb,amsthm}
\usepackage[final]{graphicx}

\newtheorem{thm}{Theorem}

\newtheorem{lem}{Lemma}
\newtheorem{prop}{Proposition}
\newtheorem{rem}{Remark}
\newtheorem{df}{Definition}

\newcommand{\RM}{\mathbb{R}}
\newcommand{\ZM}{\mathbb{Z}}

\newcommand{\CM}{\mathbb{C}}
\newcommand{\MM}{\,\mbox{\bf M}}
\newcommand{\WM}{\,\mbox{\bf W}}

\newcommand{\HM}{\,\mbox{\bf H}}
\newcommand{\UM}{\,\mbox{\bf U}}

\newcommand{\IM}{\;\mbox{\bf I}}
\newcommand{\tr}{\,\mbox{\rm tr}}

\newcommand{\cof}{\operatorname{cof}}

\newcommand{\cn}{\operatorname{cn}}
\newcommand{\dn}{\operatorname{dn}}

\newcommand{\spec}{\operatorname{spec}}
\newcommand{\sech}{\operatorname{sech}}

\title{On the Stability of Periodic Solutions of the Generalized Benjamin-Bona-Mahony Equation}
\author{Mathew A. Johnson\thanks{Department of Mathematics, Indiana University, Bloomington, IN 47405 U.S.A.}}

\begin{document}
\bibliographystyle{plain}
\maketitle

\begin{abstract}
We study the stability of a four parameter family of spatially periodic traveling wave solutions of the generalized
Benjamin-Bona-Mahony equation to two classes of perturbations: periodic perturbations with the same periodic
structure as the underlying wave, and long-wavelength localized perturbations.
In particular, we derive necessary conditions for spectral instability to perturbations to both
classes of perturbations by deriving appropriate asymptotic expansions of the
periodic Evans function, and we outline a nonlinear stability theory to periodic perturbations
based on variational methods which effectively extends our periodic spectral stability results.
\end{abstract}



\pagestyle{myheadings}
\thispagestyle{plain}


\markboth{Mathew A. Johnson}{Stability of Periodic gBBM Waves}
\section{Introduction}

In this paper, we consider the generalized Benjamin-Bona-Mahony (gBBM) equation
\begin{equation}
u_t-u_{xxt}+u_x+\left(f(u)\right)_{x}=0,\label{gbbm}
\end{equation}
where $x\in\RM$, $t\in\RM$, and $f(\cdot)\in C^2(\RM)$ is a prescribed nonlinearity.  In particular,
we will be most interested in the case of a power-law nonlinearity $f(u)=u^{p+1}$ as this is when
our results are most explicit.  Notice that when $f(u)=u^2$, \eqref{gbbm} is precisely the
Benjamin-Bona-Mahony equation (BBM), or the regularized long-wave equation, which arises
as an alternative model to the well known Korteweg-de Vries equation (KdV)
\[
u_t-u_{xxx}+u u_x=0
\]
as a description of gravity water waves in the long-wave regime (see \cite{BBM,DHP}).  In applications,
various other nonlinearities arise which facilitates our consideration of the
generalized BBM equation.  For suitable nonlinearities, equation \eqref{gbbm}
admits traveling wave solutions of the form $u(x,t)=u(x-ct)$ with wave speed $c>1$ which are
are either periodic or asymptotically constant.  Solutions which are asymptotically constant
are known as the solitary waves, and they correspond to either homoclinic or heteroclinic
orbits of the traveling wave ODE
\begin{equation}
cu_{xxt}-(c-1)u_x+\left(f(u)\right)_x=0\label{travelode}
\end{equation}
obtained from substituting the traveling wave ansatz into \eqref{gbbm}.  The stability of such
solutions to localized, i.e. $L^2(\RM)$, perturbations is well known \cite{BS,PW1,MW1}: the solitary waves form a one
parameter family of traveling wave solutions of the gBBM which can be indexed by the wave speed $c>1$.
A given solitary wave solution $u_{c_0}(x-c_0t)$ of the  is nonlinearly (orbitally) stable if the
so called momentum functional
\[
\mathcal{N}(c):= \frac{1}{2}\int_\RM \left(u_c(x)^2+u_c'(x)^2\right)dx
\]
is an increasing function at $c_0$, i.e. if $\frac{d}{dc}\mathcal{N}(c)|_{c=c_0}>0$, and is exponentially
unstable if $\frac{d}{dc}\mathcal{N}(c)|_{c=c_0}$ is negative.  In the special case of a power-nonlinearity
$f(u)=u^{p+1}$, it follows that such waves are nonlinearly stable if $1\leq p\leq 4$, while for $p>4$ there
exists a critical wavespeed $c(p)$ such that the waves are nonlinearly unstable for $1< c<c(p)$ and
stable for $c>c(p)$.

In \cite{PW1}, Pego and Weinstein found the mechanism for the instability of the solitary waves
to be as follows: linearizing the traveling gBBM equation
\begin{equation}
u_t-u_{xxt}+cu_{xxx}-(c-1)u_x+f(u)_x=0\label{travelgbbm}
\end{equation}
about a given solitary wave solution $u_{c}$ of \eqref{travelode}
and taking the Laplace transform
in time yields a spectral problem of the form $\partial_x\mathcal{L}[u_{c}]v=\mu v$ considered on the
real Hilbert space $L^2(\RM)$.  The authors then make a detailed analysis of the Evans function $D(\mu)$, which
plays the role of a transmission coefficient familiar from quantum scattering theory: in particular, $D(\mu)$
measures intersections of the unstable manifold as $x\to-\infty$ and the stable manifold as $x\to+\infty$
of the traveling wave ODE.  As a result, if Re$(\mu)>0$ and $D(\mu)=0$, it follows that the spectral
problem $\partial_x\mathcal{L}[u_{c}]v=\mu v$ has a non-trivial $L^2(\RM)$ solution, and hence $\mu$ belongs
to the point spectrum of the linearized operator\footnote{By a standard argument, the essential spectrum
can be shown to lie on the imaginary axis, and hence any spectral instability must come from the discrete spectrum.}.
Using this machinery, Pego and Weinstein were able to prove that ${\rm sign}(D(\mu))=+1$ as $\mu\to+\infty$ and
\[
D(\mu)=C\left(\frac{d}{dc}\mathcal{N}(\omega)\big{|}_{\omega=c}\right)\mu^2+\mathcal{O}(|\mu|^3)
\]
for $|\mu|\ll 1$ and some constant $C>0$.  Thus, if $\mathcal{N}(c)$ is decreasing at $c$, then by continuity there must exist
a real $\mu^*>0$ such that $D(\mu^*)=0$, which proves exponential instability of the underlying traveling wave.

The focus of the present work concerns the stability properties of spatially periodic traveling wave solutions
of \eqref{gbbm} and in contrast to its solitary wave counterpart, relatively little is known in this context.
Most results in the periodic case falls into one of the following
two categories: spectral stability with respect to localized or bounded perturbations
\cite{BD,BrJ,HK,Haragus}, and nonlinear (orbital) stability
with respect to periodic perturbations\cite{AP,APBS,BrJK,DK,MJ1}.  The spectral stability results rely on a detailed analysis of the
spectrum of the linearized operator in a given Hilbert space representing
the class of admissible perturbations: in our case, we will consider
both $L^2(\RM)$, corresponding to localized perturbations, and $L^2_{\rm per}([0,T])$,
corresponding to $T$-periodic perturbations where $T$ is the period of the underlying wave.
Notice that unstable spectrum in $L^2_{\rm per}([0,T])$ represents a high-frequency instability and hence manifests itself
in the short time instability (local well-posedness) of the underlying solution, while
unstable spectrum near the origin corresponds to instability to long wavelength (low-frequency) perturbations, or slow modulations,
and hence manifests itself in the long time instability (global well-posedness) of the underlying
solution.  Notice that on a mathematical level, the origin in the spectral
plane is distinguished by the fact that the traveling wave ordinary differential equation
\eqref{travelode} is completely integrable.  Thus, the tangent space to the manifold of periodic traveling wave solutions
can be explicitly computed, and the null space of the linearized operator can be built
up out of a basis for this tangent space.

Our analysis of the linearized spectral problem on $L^2_{\rm per}([0,T])$ parallels that
of the solitary wave theory described above, in that we compare the large real $\mu$ behavior of the corresponding
periodic Evans function to the local behavior near the origin, thus deriving a sufficient condition
for instability.  The stability analysis to arbitrary localized perturbations is more delicate and
follows the general modulational theory techniques of Bronski and Johnson \cite{BrJ}: unlike the
solitary wave case, the spectrum of the linearized operator about a periodic wave has purely continuous $L^2$ spectrum
and hence any spectral instability must come from the essential spectrum.  As a result, there are relatively few
results in this case.  In the well known work of Gardner \cite{Gard}, it is shown that periodic
traveling wave solutions of \eqref{gbbm} of sufficiently long wavelength are exponentially unstable whenever
the limiting homoclinic orbit (solitary wave) is unstable.  The mechanism behind this instability is the existence
of a ``loop" of spectrum in the neighborhood of any unstable eigenvalue of the limiting solitary wave.  More recently,
H\v{a}r\v{a}gu\c{s} carried out a detailed spectral stability analysis in the case
of power-nonlinearity $f(u)=u^{p+1}/(p+1)$ for waves sufficiently close to the constant state $u=\left((p+1)(c-1)\right)^{1/p}$
determining when the $L^2$ spectrum of the linearized operator is confined to the imaginary axis.  We, on the other hand,
consider arbitrary periodic traveling waves with essentially arbitrary nonlinearity: as a result of this
level of generality we are unable to make conclusive spectral stability statements but instead can
only determine $L^2$-spectral stability near the origin: this is determined by an asymptotic analysis
of the periodic Evans function near the origin resulting in a modulational instability index which determines
the local normal form of the spectrum.  However, we find that this analysis yields
quite a bit of information about the spectrum of the underlying wave.

The orbital stability results rely on the now familiar energy
functional techniques or Grillakis, Shatah, and Strauss \cite{GSS}
used in the study of solitary type solutions.  
This program has recently been carried out by Johnson \cite{MJ1} in the case of periodic traveling
waves of the generalized Korteweg-de Vries equation.  The corresponding theory for the gBBM equation
is nearly identical to the analysis in \cite{MJ1}, and hence will only be outlined in this work.

Finally, we study the behavior of the above indices in a long-wavelength limit in the case
of a power-nonlinearity $f(u)=u^{p+1}/(p+1)$.  For such non-linearities, we gain an additional
scaling in the wave speed which allows explicit calculations of the leading order
terms in this stability index.  A naive guess would be that the value
of the orientation index would converge to the solitary wave stability index.  This is not the case,
however, since the convergence of long-wavelength periodic waves to solitary waves is non-uniform, implying
such a limit is highly singular.  What is true is that the {\it sign} of the finite-wavelength instability
index converges to the {\it sign} of the solitary wave stability index.  Since this index was derived from
an orientation index calculation, only its sign matters and thus we show that this stability index
is in some sense the correct generalization of the stability index used in the study of solitary waves of \eqref{gbbm}.
Moreover, the sign of the modulational stability index converges to the same sign as the solitary wave index,
implying that periodic waves of \eqref{gbbm} in a neighborhood of a solitary wave are modulationally unstable
if and only if the nearby solitary wave unstable.  Notice this does not follow directly from the work of Gardner
described above: there it is proved that there exists a ``loop" of spectrum in the neighborhood of the origin which is the
{it continuous} image of a circle, but to our knowledge it has never been proved that this map is injective.

The outline for this paper is as follows.  In section 2 we review the basic properties of the periodic
traveling wave solutions of \eqref{gbbm}, and in section 3 we review the basic properties of the
periodic Evans function utilized throughout this work.  In section 4, we begin our analysis by considering
stability of the $T$-periodic traveling wave to $T$-periodic perturbations.  In particular, we first
determine the orientation index, which provides sufficient information for spectral instability
to such perturbations, and then discuss how this index plays into the nonlinear stability theory.
In section 5, we conduct our modulational instability analysis.  In section 6 we analyze
the results of sections 4 and 5 in a solitary wave limit, thus extending
the well known results of Garnder.  Finally, we close in section 7 with a brief discussion and closing remarks.

\section{Properties of the Periodic Traveling Waves}

In this section, we describe the basic properties of the periodic traveling waves of the gBBM equation \eqref{gbbm}.
For each $c>1$, a traveling wave is a solution of the traveling wave ODE \eqref{travelode},
i.e. they are stationary solutions of \eqref{gbbm} in a moving coordinate frame defined by $x-ct$.
Clearly \eqref{travelode} defines a Hamiltonian ODE and can be reduced to quadrature: in particular, the
traveling waves satisfy the relations
\begin{align}
cu_{xx}-(c-1)u+f(u)&=a\nonumber\\
\frac{c}{2}u_{x}^2-\left(\frac{c-1}{2}\right)u^2+F(u) &= au+E\label{quad2}
\end{align}
where $a$ and $E$ are real constants of integration and $F'=f$ with $F(0)=0$.
Notice that the solitary waves correspond to $a=0$ and $E$ fixed by the asymptotic values of the solution.
However, the parameters $a$ and $E$ are free parameters: we must only require
that the effective potential
\[
V(u;a,c):=F(u)-\frac{c-1}{2}u^2-au
\]
have a non-degenerate local minimum (see Figure \ref{phasespace}).  Note that this places
restrictions on the allowable parameter regime for our problem: we will always assume we are in the interior
of this region, and that the roots $u_{\pm}$ of the equation $V(x;a,c)=E$ are simple and such that $V(x;a,c)<E$
for $x\in(u_{-},u_{+})$.  In particular, this guarantees the classical ``turning points" $u_{\pm}$ are $C^1$ functions
of the parameters $a$, $E$, and $c$.
Thus, the periodic solutions of \eqref{travelode}
form a four parameter family of solutions $u(x+x_0;a,E,c)$ while the solitary waves form
a codimension two subset.  Notice however that the translation invariance, corresponding
to the parameter $x_0$ is not essential to our theory and can be modded out.  Hence,
we consider the periodic traveling wave solutions of \eqref{gbbm} as a three parameter
family of the form $u(x;a,E,c)$.
The partial differential equation \eqref{gbbm} has, in general, the three conserved quantities
\begin{eqnarray}
M&=&\int_0^T\left(u-u_{xx}\right)dx\nonumber\\
P&=&\frac{1}{2}\int_0^T\left(u^2+u_x^2\right)dx\label{momentum}\\
H&=&\int_0^T\left(\frac{1}{2}u^2+F(u)\right)dx\nonumber
\end{eqnarray}
which correspond to the mass, momentum, and Hamiltonian (energy) of the solution, respectively.  These three
quantities are considered as functions of the traveling wave parameters $a$, $E$, and $c$ and their
gradients with respect to these parameters will play an important role in the foregoing analysis.  It
is important to notice that when restricted to the four-parameter family of periodic traveling wave
solutions of \eqref{gbbm}, the mass can be represented as $M=\int_0^Tu\;dx$.  Since all our results
concern this four-parameter family, we will always work with this simplified expression for the mass
functional.

\begin{figure}
\centering
\includegraphics[scale=0.85]{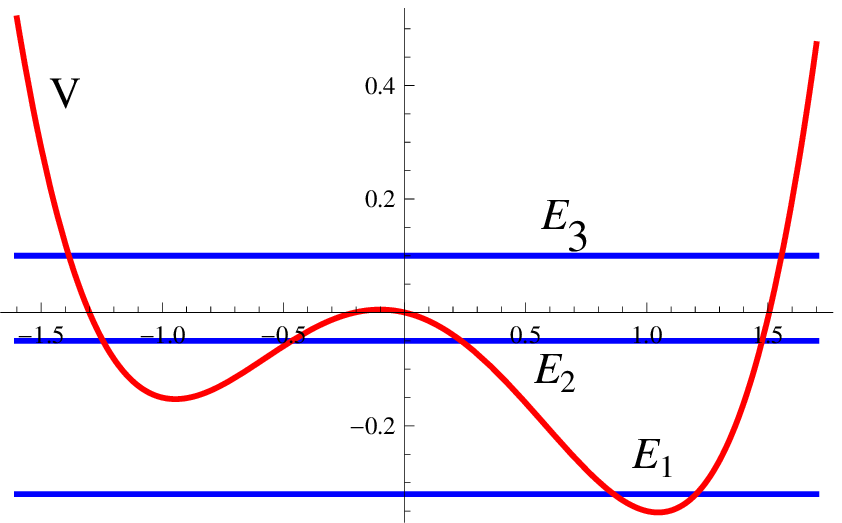}~~~~~
\includegraphics[scale=0.6]{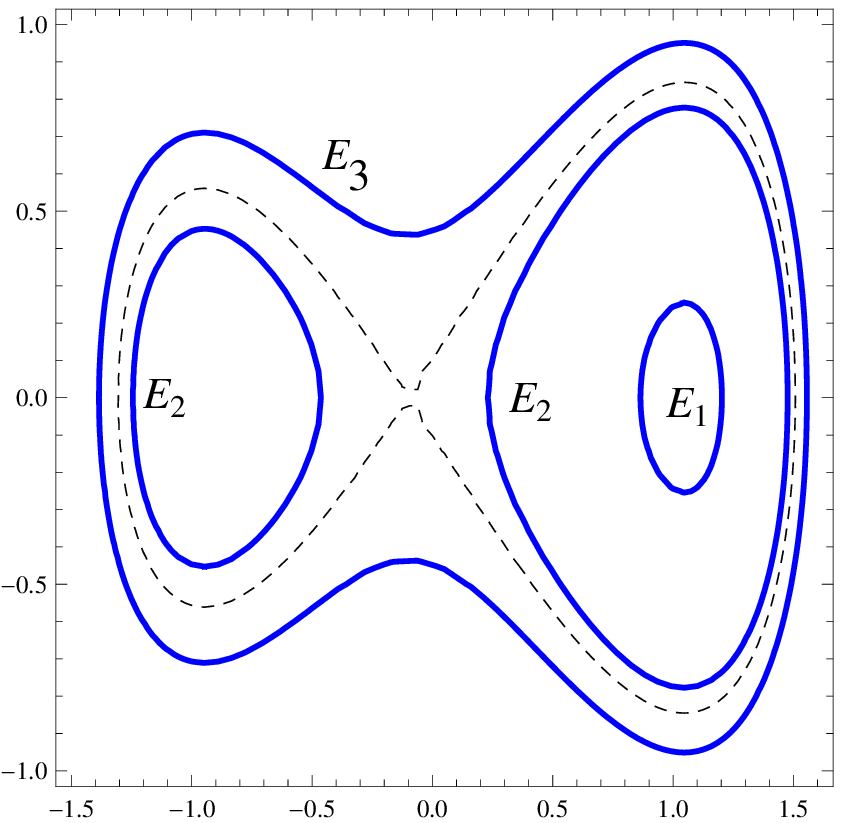}
\caption{(Left) A plot of the effective potential energy $V(x;0.1,2)$ for the modified BBM equation ($f(u)=u^3$), as well
as three energy levels $E_1=-0.32$, $E_2=-0.05$, and $E_3=0.1$. (Right) Plots in phase space
of the solutions $u(x;0.1,E_j,2)$ corresponding to the three energy levels on the left.  Notice
that those solutions corresponding to energy levels $E_1$ and $E_2$ are bounded
by a homoclinic orbit in phase space (given by the thin dashed line).  Moreover, notice that $E_2$ corresponds to two distinct
periodic traveling wave solutions: however, these can be clearly distinguished by their initial values which we have
chosen to mod out in our theory.}
\label{phasespace}
\end{figure}


As is standard, one can use equation \eqref{quad2} to express the period of the periodic wave $u$ as
\begin{equation*}
T(a,E,c)=2\sqrt{c}\int_{u_{-}}^{u_{+}}\frac{du}{\sqrt{2\left(E-V(u;a,c)\right)}}.
\end{equation*}
The above interval can be regularized at the square root branch
points $ u_{\pm}$ by the a standard procedure (see \cite{BrJ} for example) and hence represents a $C^1$ function
of $(a,E,c)$.
Similarly, the mass and momentum can be expressed as
\begin{eqnarray*}
M(a,E,c)&=&2\sqrt{c}\int_{u_{-}}^{u_{+}}\frac{u\;du}{\sqrt{2\left(E-V(u;a,c)\right)}}\\
P(a,E,c)&=&\int_{u_{-}}^{u_{+}}\left(\frac{\sqrt{c}~u^2}{\sqrt{2\left(E-V(u;a,c)\right)}}+\sqrt{\frac{2}{c}}\sqrt{\left(E-V(u;a,c)\right)}\right)du
\end{eqnarray*}
and can be regularized as above.  In particular, it follows that one can differentiate these
functionals restricted to the periodic wave $u(x;a,E,c)$ with respect to the parameters $(a,E,c)$.  The gradients
of these quantities will play an important role in the subsequent theory.  

It is useful to notice the following connection to the classical mechanics corresponding to
the traveling wave ODE.  The classical action in the sense of action angle variables is
given by
\begin{equation}
K(a,E,c)=\oint u_xdu=2\sqrt{\frac{2}{c}}\int_{u_{-}}^{u_{+}}\sqrt{\left(E-V(u;a,c)\right)}du.\label{action}
\end{equation}
While $K$ is not itself conserved, it does provide a useful generating function for the
conserved quantities of \eqref{gbbm}.  Indeed, it is clear the classical action satisfies the relation
\begin{eqnarray*}
\nabla_{a,E,c}K(a,E,c)=\left<\frac{1}{c}M(a,E,c),\frac{1}{c}T(a,E,c),\frac{1}{c}\left(P(a,E,c)-K(a,E,c)\right)\right>,
\end{eqnarray*}
where $\nabla_{a,E,c}:=\left<\partial_a,\partial_E,\partial_c\right>$,
which immediately establishes several useful identities between the gradients of $T$, $M$, and $P$.
For example, it follows that $T_a=M_E$ and $T_c=P_E$.

Finally, we make a few notes on notation.  Throughout the forthcoming analysis, various
Jacobians of maps from the traveling wave parameters to the period and conserved quantities
of the gBBM flow will become important.  We adopt the following Possion bracket style notation
\[
\{g,h\}_{x,y} = \frac{\partial(g,h)}{\partial(x,y)}
\]
for the Jacobian determinants with the analogous notation for larger determinants
\[
\{g,h,j\}_{x,y,z}=\frac{\partial(g,h,j)}{\partial(x,y,z)}.
\]
Notice that the nonvanishing of such quantities encodes geometric information about the
underlying manifold of periodic traveling wave solutions of \eqref{gbbm}: more will be said on
this in the coming sections.


\section{The Periodic Evans Function}

We now begin our stability analysis of a $T=T(a,E,c)$-periodic traveling wave $u(x;a,E,c)$ by considering a solution to the
partial differential equation \eqref{gbbm}, i.e. a stationary solution of \eqref{travelgbbm}, of the form
\[
\psi(x,t)=u(x;a,E,c)+\varepsilon v(x,t)+\mathcal{O}(\varepsilon^2)
\]
where $|\varepsilon|\ll 1$ is considered as a small perturbation parameter.  Substituting this into \eqref{gbbm} and collecting
terms at $\mathcal{O}(\varepsilon)$ yields the linearized equation $J\mathcal{L}[u]v=\mathcal{D}v_t$ where $J=\partial_x$,
$\mathcal{L}[u]=-c\partial_x^2+(c-1)-f'(u)$, and $\mathcal{D}=1-\partial_x^2$.  Since this linearized equation is autonomous
in time, we may seek separated solutions of the form $v(x,t)=e^{\mu t}v(x)$, which yields the spectral problem
\begin{equation}
J\mathcal{L}v=\mu\mathcal{D}v\label{linear}
\end{equation}
Throughout this paper, we consider the above operators as acting on $L^2(\RM)$ 
corresponding to spatially localized perturbations, or on $L^2_{\rm per}([0,T])$ 
corresponding to $T$-periodic perturbations.  In both cases, the
operator $\mathcal{D}$ is a positive operator and is hence invertible and hence 
\eqref{linear} can be written as a spectral problem for a linear operator:
\[
\mathcal{A}v=\mu v,\;\;\;\mathcal{A}=\mathcal{D}^{-1}J\mathcal{L}.
\]
As $\mathcal{L}$ has is a differential operator with periodic coefficients, the natural setting to study
the $L^2(\RM)$ spectrum of the operator $\mathcal{A}$ is that of Floquet theory.  We begin with
the following standard definition.

%
%
%
%
%
%
%
%
%
%
%
%

\begin{df}
The monodromy operator $M(\mu)$ is defined to be the period map
\[
\MM(\mu)=\Phi(T;\mu)
\]
where $\Phi(x;\mu)$ satisfies the first order system
\begin{equation}
\Phi(x;\mu)_x=\HM(x,\mu)\Phi(x;\mu)\label{system}
\end{equation}
subject to the initial condition $\Phi(0)={\bf I}$, where ${\bf I}$ is the $3\times 3$ identity matrix and
\[
\HM(x,\mu)=\left(
           \begin{array}{ccc}
             0 & 1 & 0 \\
             0 & 0 & 1 \\
             -\frac{1}{c}\left(\mu+u_xf''(u)\right) & \frac{1}{c}\left(c-1-f'(u)\right) & \frac{\mu}{c} \\
           \end{array}
         \right).
\]
\end{df}

It now follows from an easy calculation that \eqref{linear} has no point spectrum in $L^2(\RM)$.  
Indeed, suppose $\Psi$ is a vector solution of \eqref{system} corresponding to a non-trivial $L^2(\RM)$ eigenfunction
of $\mathcal{A}$ with eigenvalue $\mu$.  From the definition of the monodromy operator, we have
\[
\Psi(NT)=\MM(\mu)^N\Psi(0)
\]
for any $N\in\ZM$.  
It follows that $\Psi$ can be at most bounded on $\RM$ and must not decay as $N\to\pm\infty$.
This observation leads one to the following definition.

\begin{df}
We say $\mu\in\spec(\mathcal{A})$ if there exists a non-trivial bounded function $\psi$ such that $\mathcal{A}\psi=\mu\psi$ or,
equivalently, if there exists a $\lambda\in S^1$ such that
\[
\det(\MM(\mu)-\lambda \IM)=0.
\]
Following Gardner \cite{Gard} we define the periodic Evans function $D:\CM\times\CM\to\CM$ to be
\[
D(\mu,\lambda)=\det(\MM(\mu)-\lambda\IM).
\]
Finally, we say the periodic solution $u(x;a,E,c)$ is spectrally stable if $\spec(\mathcal{A})$ does not intersect
the open right half plane.
\end{df}

\begin{rem}
First, notice by the Hamiltonian nature of \eqref{linear}, $\spec(\mathcal{A})$ is symmetric with respect to reflections about
the real and imaginary axis.  Thus, spectral stability occurs if and only if $\spec(\mathcal{A})\subset\RM i$.

Secondly, since we are interested primarily with the roots of $D(\mu,\lambda)$ for $\lambda$ on the unit
circle, we will frequently work with the function $D(\mu,e^{i\kappa})$ for $\kappa\in\RM$, which is
actually the function considered by Gardner.
\end{rem}

It follows that we can parameterize the continuous $L^2(\RM)$ spectrum of the operator $\mathcal{A}$
by the Floquet parameter $\kappa$:
\[
\spec(\mathcal{A})=\bigcup_{\kappa\in[-\pi,\pi)}\{\mu\in\CM:D(\mu,e^{i\kappa})=0\}.
\]
In particular, the zero's of the function $D(\mu,e^{i\kappa})$ for a fixed $\kappa\in\RM$ correspond
to the $L^2_{\rm per}([0,T])$-eigenvalues of the operator resulting from the map
$\partial_x\mapsto\partial_x+i\kappa$ applied to \eqref{linear}, and hence the continuous
spectrum of $A$ can be parameterized by a one-parameter family of eigenvalue problems.
By a standard result of Gardner, if $D(\mu_0,e^{i\kappa_0})=0$, then the multiplicity
of $\mu_0$ as a periodic eigenvalue of the corresponding linear operator is precisely the
multiplicity of $\mu_0$ as a root of the Evans function.
As we will see below, the integrable structure of \eqref{travelode} implies that the function $D(\mu,1)$ has
a zero of multiplicity (generically) three at $\mu=0$.  For small $\kappa$ then, there will be
in general three branches $\mu_j(\kappa)$ of roots of $D(\mu_j(\kappa),e^{i\kappa})$ which bifurcate
from the origin.  Assuming these branches are analytic\footnote{In general, for each $j$, the theory of branching solutions of non-linear
equations guarantees the existence of a natural number $m_j$ such that $\mu_j(\cdot)$ is an analytic function of $\kappa^{1/m_j}$.
As we will see in our case, the Hamiltonian nature of the linearized operator $\mathcal{A}$ assures that $m_j=1$, and hence
the roots are in fact analytic functions of the Floquet parameter.}
in $\kappa$, it follows that a necessary condition for spectral stability is thus
\begin{equation}
\frac{\partial}{\partial \kappa}\mu_j(\kappa)\big{|}_{\kappa=0}\in\RM i.\label{firstordercondition}
\end{equation}
This naturally leads to the use of perturbation methods in the study of the spectrum of $\mathcal{A}$ near the origin, i.e.
modulational instability analysis of the underlying traveling wave.  As we will see, the first order
terms of a Taylor series expansion of the three branches $\mu_{j}(\kappa)$ can be encoded as roots
of a cubic polynomial, and hence spectral stability is determined by the sign of the associated 
discriminant.  Moreover, it follows by the Hamiltonian structure of \eqref{linear} that in fact $\sigma(\mathcal{A})\subset\RM i$
if \eqref{firstordercondition} holds and the roots of the cubic polynomial are distinct.

We conclude this section by reviewing some basic global features of the spectrum of the linearized operator $\mathcal{A}$
which are useful in a local analysis near $\mu = 0$. We also state some important properties of the Evans
function $D(\mu,\lambda)$ which are vital to the foregoing analysis.

\begin{prop}\label{globalspec}
The $L^2(\RM)$-spectrum of the operator $\mathcal{A}$ has the following properties:
\begin{enumerate}
\item[(i)] There are no isolated points of the spectrum.  In particular, the spectrum generically consists of
piecewise smooth arcs.
\item[(ii)] The entire imaginary axis is contained in the spectrum, i.e. $i \RM \subset \spec(\mathcal{A}).$
\end{enumerate}
Moreover, the Evans function $D(\mu,\lambda)$ satisfies the following:
\begin{enumerate}
\item[(iii)] $D(\mu,0)=e^{\mu T/c}$.
\item[(iv)] $D(\mu,\lambda)=\det(\MM(\mu)-\lambda {\bf I})=-\lambda^3 + a(\mu) \lambda^2 - a(-\mu)e^{\mu T/c}\lambda +e^{\mu T/c} $ with $a(\mu) = \tr(\MM(\mu)).$
\end{enumerate}
\end{prop}

\begin{proof}
The first claim, that the spectrum is never discrete, follows from
a basic lemma in the theory of several complex variables: namely
that, if for fixed $\lambda^*$ the function
$D(\mu,\lambda^*)$ has a zero of order $k$ at $\mu^*$ and is holomorphic in a
polydisc about $(\mu^*,\lambda^*)$ then there is some smaller
polydisc about $(\mu^*,\lambda^*)$ so that for every $\lambda$ in a disc about
$\lambda^*$ the function $D(\mu,\lambda)$ (with $\lambda$ fixed) has $k$ roots
in the disc $|\mu-\mu^*|<\delta$. For details see the text of Gunning\cite{Gun}. It is clear from the implicit function
theorem that $\mu$ is a smooth function of
$\lambda$ as long as $\frac{\partial D}{\partial \mu} = \tr(\cof^t(\MM(\mu)-\lambda {\bf I})\MM_\mu^\prime)\neq 0,$
where $\cof$ represents the standard cofactor matrix.

Claim $(iii)$ follows from Abel's formula and the fact that $c\tr(\HM(x,\mu))=\mu$, where $\HM(x,\mu)$
is as in \eqref{system}.


Next, we prove claim $(iv)$.  Since $\mathcal{A}v=\mu v$ is invariant
under the transformation $x\mapsto -x$ and $\mu\mapsto -\mu$, we have $M(\mu)\sim M(-\mu)^{-1}$.  If we define $a(\mu)$
as above and $b(\mu)$ such that
\[
\det[\MM(\mu)-\lambda {\bf I}]=-\lambda^3+a(\mu)\lambda^2+b(\mu)\lambda +e^{\mu T/c},
\]
it follows that
\begin{eqnarray*}
\det[\MM(\mu)-\lambda {\bf I}]&=&\det[{\MM}^{-1}(-\mu)-\lambda {\bf I}]\\
&=&-\lambda^3\det[{\MM}^{-1}(-\mu)]\det[\MM(-\mu)-\lambda^{-1}]\\
&=&-e^{\mu T/c}\lambda^3\left(-\lambda^{-3}+a(-\mu)\lambda^{-2}+b(-\mu)\lambda^{-1}+e^{-\mu T/c}\right)\\
&=&-\lambda^3-e^{\mu T/c}b(-\mu)\lambda^2-e^{\mu T/c}a(-\mu)\lambda+e^{\mu T/c}.
\end{eqnarray*}
Therefore, $b(\mu)=-e^{\mu T/c}a(-\mu)$ as claimed.

Claim $(ii)$ now follows from a symmetry argument.  Since $a(\mu)$ is real on the real axis, it follows by Schwarz reflection
that for $\mu\in\RM i$ we have $a(\overline{\mu})=\overline{a(\mu)}$.  For $\mu\in\RM i$ then the Evans function takes the form
\[
D(\mu,\lambda)=-\lambda^3+a(\mu)\lambda^2-e^{\mu T}\overline{a(\mu)}\lambda+e^{\mu T/c}.
\]
It follows that
\[
D(\mu,\lambda)=-\lambda^3 e^{\mu T/c}\overline{D(\mu,\overline{\lambda}^{-1})}
\]
so that the roots of $D(\mu,\lambda)$ for a fixed $\mu\in\RM i$ are symmetric about the unit circle.
Since there must be three such roots, one must live on the unit circle and hence $\mu\in\spec(\mathcal{A})$ as claimed.
\end{proof}

\section{Periodic Instabilities: Spectral and Nonlinear Stability Results}

In this section, we make a detailed analysis of the stability of a given periodic traveling wave
solution of \eqref{gbbm} to perturbations with the same periodic structure: such
perturbations correspond to $\lambda=1$ in the above theory.
We begin by a spectral stability analysis, and then conclude with a brief discussion
of a nonlinear (orbital) stability result.

\subsection{Periodic Spectral Instabilities}

Let $u(x;a,E,c)$ be a $T$-periodic traveling wave solution of \eqref{gbbm}.  Considering
the spectral stability of such a solution of perturbations which are $T$-periodic is
equivalent to studying the spectrum of the linear operator $\mathcal{A}$
on the real Hilbert space $L^2_{\rm per}([0,T])$.  We begin with the following lemma.

\begin{lem}\label{kernel}
Let $u(x;a,E,c)$ be the solution of the traveling wave equation \eqref{travelode} satisfying $u(0;a,E,c)=u_{-}$
and $u_x(0;a,E,c)=0$.  A basis of solutions to the first order system
\[
Y_x=\HM(x,0)Y
\]
is given by
\[
Y_1(x)=\left(
               \begin{array}{c}
                 cu_x(x;a,E,c) \\
                 cu_{xx}(x;a,E,c) \\
                 cu_{xxx}(x;a,E,c) \\
               \end{array}
             \right),
\;\;
Y_2(x)=\left(
               \begin{array}{c}
                 cu_a(x;a,E,c) \\
                 cu_{ax}(x;a,E,c) \\
                 cu_{axx}(x;a,E,c) \\
               \end{array}
             \right),
\;\;
Y_3(x)=\left(
               \begin{array}{c}
                 cu_E(x;a,E,c) \\
                 cu_{Ex}(x;a,E,c) \\
                 cu_{Exx}(x;a,E,c) \\
               \end{array}
             \right).
\]
Moreover, a particular solution to the inhomogeneous problem
\[
Y_x=\HM(x,0)Y+W
\]
where $W=\left(0,0,c\mathcal{D}u_x\right)$ is given by
\[
Y_4(x)=\left(
               \begin{array}{c}
                 -cu_c(x;a,E,c) \\
                 -cu_{cx}(x;a,E,c) \\
                 -cu_{cxx}(x;a,E,c) \\
               \end{array}
             \right).
\]
\end{lem}

\begin{proof}
This is easily verified by differentiating \eqref{travelode} with respect $x$ and the parameters $a$, $E$, and $c$.
\end{proof}

By the above lemma, three linearly independent solutions of the differential equation
\[
\partial_x\mathcal{L}v=0
\]
are given by the functions $u_x$, $u_a$, and $u_E$ above.  Notice that here we are
considering the formal operators with out any reference to boundary conditions.
In order to understand the structure of the periodic null-space of the operator $\mathcal{A}$,
we now use Lemma \ref{kernel} to express the solution matrix in this basis at $x=T$ and $x=0$.

Notice by hypothesis, for any $a,E,c$ the solution $u$ satisfies
\begin{eqnarray}
u(0;a,E,c)&=&u_{-}=u(T;a,E,c)\label{initial1}\\
u_x(0;a,E,c)&=&0=u_x(T,a,E,c)\label{initial2}\\
u_{xx}(0;a,E,c)&=&-\frac{1}{c}V'(u_{-};a,c)=u_{xx}(0;a,E,c)\label{initial3}
\end{eqnarray}
and, moreover, it follows from \eqref{travelode} that $u_{xxx}(0;a,E,c)=0$.  Defining $\UM(x,0)=[Y_1(x),Y_2(x),Y_3(x)]$
where $Y_1$, $Y_2$, $Y_3$ are vector functions corresponding to the solutions $u_x$, $u_a$, and $u_E$, respectively,
it follows by differentiating the above relations that
\begin{equation}
\UM(0,0)=
\left(
  \begin{array}{ccc}
    0 & c\frac{\partial u_{-}}{\partial a} & c\frac{\partial u_{-}}{\partial E} \\
    -V'(u_{-}) & 0 & 0 \\
    0 & 1+\left(c-1-f'(u_{-})\right)\frac{\partial u_{-}}{\partial a} & \left(c-1-f'(u_{-})\right)\frac{\partial u_{-}}{\partial E} \\
  \end{array}
\right).\label{matrixsolutioninitial}
\end{equation}
Differentiating the relation $E=V'(u_{-};a,c)$ with respect to $E$ gives
\[
\det(\UM(0,0))=-cV'(u_{-})\frac{\partial u_{-}}{\partial E}=-c,
\]
and hence these solutions are linearly independent at $x=0$, and hence for all $x$.  
Moreover, using the chain rule to differentiate \eqref{initial1}-\eqref{initial3}, we see
the matrix $\UM(T;0)$ is given by
\[
\UM(T,0)=\UM(0,0)+
\left(
  \begin{array}{ccc}
    0 & 0 & 0 \\
    0 & V '(u_{-})T_a & V '(u_{-})T_E \\
    0 & 0 & 0 \\
  \end{array}
\right).
\]
It follows that $\UM(T,0)-\UM(0,0)$ is a rank one matrix.  
%

Our immediate goal is to relate this information to the structure of the periodic Evans function
$D(\mu,e^{i\kappa})$.  As we will see, the fact that $\MM(0)-{\bf I}$ is rank one allows for significant
simplifications in the perturbation calculations.  Without this fact, one would have to compute variations
in the vector solutions at $\mu=0$ to an extra degree, which would force the use of multiple applications of
variation of parameters along the same null direction.  In particular, we have the following lemma.

\begin{lem}\label{Evans1}
For $|\mu|\ll 1$, the periodic Evans function satisfies
\[
D(\mu,1)=-\{T,M,P\}_{a,E,c}\mu^3+\mathcal{O}(|\mu|^3).
\]
\end{lem}

\begin{proof}
Let $w_i(x;\mu)$, $i=1,2,3$, be three linearly independent solutions of \eqref{system}, and let
$\WM(x,\mu)$ be the solution matrix with columns $w_i$.  Expanding the above solutions in powers
of $\mu$ as
\[
w_i(x,\mu)=w_i^0(x)+\mu w_i^1(x)+\mu^2 w_i^2(x)+\mathcal{O}(|\mu|^3)
\]
and substituting them into \eqref{system}, the leading order equation
becomes
\[
\frac{d}{dx}w_i^0(x)=\HM(x,0)w_i^0(x).
\]
Using Proposition \ref{kernel}, we choose $w_i^0(x)=Y_i(x)$.
The higher order terms in the above expansion yield
\begin{equation}
\frac{d}{dx}w_i^{j}(x)=\HM(x,0)w_i^{j}(x)+V_i^{j-1}(x),\;\;j\geq 1,\label{nonhomeqn}
\end{equation}
where $V_i^{j-1}=\left(0,0,-c^{-1}\mathcal{D}(w_i^{j-1})_1\right)^{t}$ and $(v)_1$ denotes the first component of the vector $v$.
Notice that for each of the higher order terms $j\geq 1$, we require $w_j^i(0)=0$.  Notice this implies
that $\WM(0,\mu)=\UM(0,0)$ in a neighborhood of $\mu=0$, where $\UM(0,0)$ is defined in \eqref{matrixsolutioninitial}.
The solution of the inhomogeneous problem is given by the variation of parameters formula
\begin{align}
w_i^j(x)&=\WM(x,0)\int_0^x \WM(s,0)^{-1}V_i^{j-1}(s)ds\nonumber\\
&=\left(
     \begin{array}{ccc}
   cu_x\int_0^x\mathcal{D}(w_i^{j-1})_1\{u,u_x\}_{a,E}\;dz-u_a\int_0^x\mathcal{D}(w_i^{j-1})_1dz+u_E\int_0^x\mathcal{D}(w_i^{j-1})_1u\;dz \\
   cu_{xx}\int_0^x\mathcal{D}(w_i^{j-1})_1\{u,u_x\}_{a,E}\;dz-u_{ax}\int_0^x\mathcal{D}(w_i^{j-1})_1dz+u_{Ex}\int_0^x\mathcal{D}(w_i^{j-1})_1u\;dz\\
   cu_{xxx}\int_0^x\mathcal{D}(w_i^{j-1})_1\{u,u_x\}_{a,E}\;dz-u_{axx}\int_0^x\mathcal{D}(w_i^{j-1})_1dz+u_{Exx}\int_0^x\mathcal{D}(w_i^{j-1})_1u\;dz\\
     \end{array}
   \right)
\label{variation}
\end{align}
for $j\geq 1$.  Notice we have used the identities $c\{u,u_x\}_{E,x}=-1$ and $c\{u,u_x\}_{x,a}=u$ extensively in the above formula,
which can be easily derived via equation \eqref{quad2}.  Indeed, differentiating \eqref{quad2} with respect to $E$
and subtracting $u_Eu_{xx}$ immediately yields the first identity.

Now, notice that it would be a daunting task to use \eqref{variation} to the specified
order needed.  However, the integrable structure of \eqref{linear} allows for an alternative, yet equivalent,
expression in the case $i=j=1$ which makes a seemingly second order calculation
come in at first order.  Indeed, in this case equation \eqref{nonhomeqn}  is equivalent
to $L_0 w_1^1=u_x$ and hence it follows from Proposition \ref{kernel} that we can choose
\[
w_1^1(x)=\left(
           \begin{array}{c}
             -cu_c \\
             -cu_{c}' \\
             -cu_{c}'' \\
           \end{array}
         \right)
         +
         \left(u_{-}+\frac{1}{c}V'(u_{-};a,c)\right)\left(
                \begin{array}{c}
                 c u_a \\
                 c u_a' \\
                 c u_a'' \\
                \end{array}
              \right)
              -\left(\frac{u_{-}^2}{2}+\frac{1}{c}V'(u_{-};a,c)\right)
              \left(
                \begin{array}{c}
                 c u_E \\
                 c u_E' \\
                 c u_E'' \\
                \end{array}
              \right)
\]
where the above constants in front of $Y_2$ and $Y_3$ are determined by the
requirement $w_1^1(0)=0$.  Thus, one can determine the second order variation of $w_1$
in $\mu$ by using \eqref{variation} to compute the {\it first} order variation of the function
$w_1^1$ defined above.  Defining $\delta\WM(\mu):=\WM(x,\mu)\big{|}_{x=0}^T$, it follows that $\WM(\mu)$ can be expanded as
\[
\left(
  \begin{array}{ccc}
    \mathcal{O}(\mu^2) & \mathcal{O}(\mu) & \mathcal{O}(\mu) \\
    \mu     
    V'(u_{-})P(u_{-})+\mathcal{O}(\mu^2)
                             & V'(u_{-})T_a & V'(u_{-})T_E  \\
    \mathcal{O}(\mu^2) & \mathcal{O}(\mu) & \mathcal{O}(\mu) \\
  \end{array}
\right)
\]
where $P(x)=-T_c+\left(x+\frac{1}{c}V'(x)\right)T_a-\left(\frac{x^2}{2}+\frac{1}{c}V'(x)\right)T_E$ and
the higher order terms are determined by \eqref{variation}.  Thus,
\[
D(\mu,1)=\det\left(\delta\WM(\mu)\WM(0,0)^{-1}\right)=\mathcal{O}(\mu^3)
\]
and in particular a straightforward calculation yields
\begin{eqnarray*}
\det\left(\delta\WM(\mu)\right)=c\{T,M,P\}_{a,E,c}\;\mu^3+\mathcal{O}(\mu^4).
\end{eqnarray*}
The proof is complete by recalling that $\det(\WM(0,0))=-c$.
\end{proof}

\begin{rem}
Notice that the formula for $D_{\mu\mu\mu}(0,1)$ differs from that derived by Bronski and Johnson \cite{BrJ}
in the case of the generalized KdV by a factor of one-half, which comes from that fact the differences
in the definitions of the momentum functionals in each work: in particular, the factor of $\frac{1}{2}$
in \eqref{momentum} is not present in \cite{BrJ}.
\end{rem}

Assuming the Jacobian $\{T,M,P\}_{a,E,c}$ is non-zero, it follows that zero is a $T$-periodic eigenvalue
of $\mathcal{A}$ of multiplicity three.  By computing the orientation index
\[
{\rm sign}\left(\{T,M,P\}_{a,E,c}\right){\rm sign}\left(D(\infty,1)\right)
\]
it is clear that we will have exponential periodic instability of the underlying wave if this index is negative.
With this in mind, we now compute the large $\mu\gg 1$ behavior of the periodic Evans function in the next lemma.

\begin{lem}
The function $D(\cdot,1)\big{|}_{\RM}:\RM\to\RM$ satisfies the following asymptotic relation:
\[
\lim_{\RM\ni\mu\to\infty} D(\mu,1)=-\infty.
\]
\end{lem}

\begin{proof}
This follows from a simple calculation.  By introducing variable change $y=|\mu|^{-1/3}x$ in \eqref{linear}, it follows that
as $\mu\to+\infty$ the monodromy operator satisfies the asymptotic relation
\begin{equation}
M(\mu)\sim \exp\left(A(\mu) T/c\right),\;\;\;|\mu|\gg 1,\label{asymptotic1}
\end{equation}
where
\[
A(\mu)=
\left(
  \begin{array}{ccc}
    0 & c & 0 \\
    0 & 0 & c \\
    -\mu & 0 & \mu \\
  \end{array}
\right).
\]
The eigenvalues $\lambda_j$ of $A(\mu)$ are rather complicated, but for $\mu$ large and real
they satisfy ${\rm Re}(\lambda_j)\sim a_j\mu$, where
\begin{align*}
a_1&=1\\
a_2&=a_3=\frac{1}{3}-\frac{1}{32^{2/3}2^{1/3}}-\frac{2^{1/3}}{62^{1/3}}.
\end{align*}
It follows that
\[
a(\mu)\sim e^{a_1\mu T/c}+e^{a_2\mu T/2c}+e^{a_3\mu T/2c}
\]
and thus, since $a_2=a_3<0<a_1<c$, the leading order term of
\[
D(\mu,1)=-1+a(\mu)-a(-\mu)e^{\mu T/c}+e^{\mu T/c}
\]
as $\mu\to+\infty$ comes from $-a(-\mu)e^{\mu T/c}$, which completes the proof.
\end{proof}

We are now able to state our first main theorem for this section.

\begin{thm}\label{thm:fw}
let $u(x;a_0,E_0,c_0)$ be a periodic solution of \eqref{travelode}.
If $\{T,M,P\}_{a,E,c}$ is negative at $(a_0,E_0,c_0)$, then the number of roots of $D(\mu,1)$ (i.e. the number of periodic eigenvalues
of $\mathcal{A}$) on the positive real axis is odd.  In particular, 
if the period
is an increasing function of energy at $(a_0,E_0,c_0)$, i.e. $T_E(a_0,E_0,c_0)>0$,
then the periodic traveling wave $u(x;a_0,E_0,c_0)$ is spectrally unstable to $T(a_0,E_0,c_0)$-periodic
perturbations if and only if $\{T,M,P\}_{a,E,c}$ is negative at $(a_0,E_0,c_0)$.
\end{thm}

\begin{proof}
By our work in the proof of Lemma \ref{Evans1}, we know that
$D(\mu,1)=\mathcal{O}(|\mu|^3)$.  Thus, if $D_{\mu\mu\mu}(0,1)<0$ for small positive $\mu$, the number $D(\mu,1)$
is negative for small positive $\mu$.  Since $D(\mu,1)$ is positive for real $\mu$ sufficiently large,
we know that $D(\pm\mu^*,1)=0$ for some $\mu^*\in\RM^*$.  Moreover, since $\mathcal{L}u_x=0$ it follows that
$\mathcal{L}$ has either one or two negative $T$-periodic eigenvalues.  Since the number of negative eigenvalues
of $\mathcal{L}$ provide an upper bound for the number of unstable $T$-periodic eigenvalues of $\mathcal{A}$
with positive real part (see Theorem 3.1 of \cite{PW1}), and since $T_E>0$ implies $\mathcal{L}$ has precisely one negative
$T$-periodic eigenvalue (see Lemma 4.1 of \cite{MJ1}) the proof is now complete.
\end{proof}

We now wish to give insight into the meaning of $\{T,M,P\}_{a,E,c}=0$ at the level of the linearized operator $\mathcal{A}$.
To this end, we consider the linearized operator $\mathcal{A}$ as acting on $L_{\rm per}^2(0,T)$, the space
of $T$-periodic $L^2$ functions on $\RM$.  To begin, we make the assumption that $\{T,M\}_{a,E}$ and
$\{T,P\}_{a,E}$ do not simultaneously vanish.  This assumption will be shown equivalent with the periodic null-space
reflecting the Jordan structure of the monodromy at the origin.  Throughout this brief
discussion, we assume that $\{T,M\}_{a,E}\neq 0$: trivial modifications are needed if $\{T,M\}_{a,E}$ vanishes
but $\{T,P\}_{a,E}$ does not.  First, define the functions
\begin{align*}
\phi_0&= \{T,u\}_{a,E},\;\;\;\;\;\;\;\;\;\;\;\;\;\;\;\;\;\;\;\;\;\;\;\;\;\psi_0=
1,\\
\phi_1&= \{T,M\}_{a,E}\;u_x,\;\;\;\;\;\;\;\;\;\;\;\;\;\;\;\;\;\;\psi_1=\int_0^x\mathcal{D}\phi_2(s)ds,\\
\phi_2&= \{u,T,M\}_{a,E,c}\;\;\;\;\;\;\;\;\;\;\;\;\;\;\;\;\;\;\;\psi_2=-\{T,M\}_{E,c}+\{T,M\}_{a,E}\mathcal{D}u,
\end{align*}
Clearly, each of these functions belong to $L^2_{\rm per}([0,T])$ and
\begin{align*}
\mathcal{A} \phi_0&=0~~~~~~~~~~~~~~~~~~~~~~~~~~~\mathcal{A}^\dag \psi_0 = 0 \\
\mathcal{A} \phi_1&=0~~~~~~~~~~~~~~~~~~~~~~~~~~~\mathcal{A}^\dag \psi_1 = \psi_2 \\
\mathcal{A} \phi_2&=-\phi_1~~~~~~~~~~~~~~~~~~~~~~~\mathcal{A}^\dag \psi_2 = 0.
\end{align*}
In particular, we have used the fact that $\mathcal{D}^{-1}(1)=1$ on $L^2_{\rm per}([0,T])$.
Thus, it follows that the periodic null space
of $\mathcal{A}$ is generated spanned by the functions $\phi_0$ and $\phi_1$.
Moreover, since
\begin{align*}
\left<\phi_0,\psi_0\right>&=\{T,M\}_{a,E}\\
\left<\phi_0,\mathcal{D}u\right>&=\{T,P\}_{a,E}
\end{align*}
the assumption that $\{T,M\}_{a,E}$ and $\{T,P\}_{a,E}$ do not simultaneously vanish implies
that $N_{\rm per}\left(\mathcal{A}^2\right)-N_{\rm per}\left(\mathcal{A}\right)={\rm span}(\phi_2)$, thus reflecting the
Jordan normal form of the period map at $\mu=0$.

Finally, we study the structure of the generalized periodic null space, and seek conditions
for which there is no non-trivial Jordan chain of length two.  By the Fredholm alternative,
such a chain exists if and only if
\[
\left<\mathcal{D}u,\phi_2\right>=\{T,M,P\}_{a,E,c}=0.
\]
Thus, the vanishing of $\{T,M,P\}_{a,E,c}$ is equivalent with a change in the generalized periodic-null space of the
linearized operator $\mathcal{A}$.  This insight has a nice relationship with formal Whitham modulation theory.
One of the big ideas in Whitham theory is to locally parameterize the periodic traveling wave solution
by the constants of motion for the PDE evolution.  The non-vanishing of certain Jacobians is precisely what
allows one to do this.  In fact, the non-vanishing of $\{T,M,P\}_{a,E,c}$ is equivalent to demanding that, locally,
the map $(a,E,c)\mapsto(T,M,P)$ have a unique $C^1$ inverse: In other words, the constants of motion for the gBBM flow
are good local coordinates for the three-dimensional manifold of periodic traveling wave solutions (up to translation).
Similarly, non-vanishing of $\{T,M\}_{a,E}$ and $\{T,P\}_{a,E}$ is equivalent to demanding that the matrix
\[
\left(
  \begin{array}{ccc}
    T_a & M_a & P_a \\
    T_E & M_E & P_E \\
  \end{array}
\right)
\]
have full rank, which is equivalent to demanding that the map $(a,E)\mapsto(T,M,P)$ (for fixed $c$) have a unique
$C^1$ inverse, i.e. two of the conserved quantities provide a smooth parametrization of the family of periodic
traveling waves of fixed wave-speed.

To summarize, the vanishing of $\{T,M,P\}_{a,E,c}$, is connected with a change in the Jordan structure of the linearized
operator $\mathcal{A}$ considered on $L^2_{\rm per}([0,T])$.  In particular, $\{T,M,P\}_{a,E,c}$ ensures the existence
of a non-vanishing Jordan piece in the generalized periodic null-space of dimension {\it exactly} one.  Moreover,
and perhaps more importantly, it guarantees that infinitesimal variations in the constants arising from reducing
the family of periodic traveling waves to quadrature are enough to generate the entire generalized periodic
null-space of the linearized operator $\mathcal{A}$: such a condition is obviously not necessary in our calculations,
but provides significant simplifications in the theory.

\subsection{Nonlinear Periodic-Stability}

We now compliment Theorem \ref{thm:fw} by considering in what sense the Jacobian $\{T,M,P\}_{a,E,c}$ affects
the nonlinear stability of a periodic traveling wave solution of \eqref{gbbm} to $T$-periodic perturbations.
Clearly, the positivity of this index is necessary for such stability but it is not clear if this is also
necessary.  Indeed, the analysis presented below allows for the possibility that a periodic traveling wave
which is spectrally stable to $T$-periodic perturbations could be nonlinearly unstable to such perturbations:
such a result would stand in stark contrast to the solitary wave theory where these two notions of stability are
equivalent (assuming the nondegeneracy condition $\frac{d}{dc}\mathcal{N}(c)\neq 0$).  More will be said on
this at the end of this section.

Such analysis has recently been carried out in the context of the generalized Korteweg-de Vries equation
by Johnson \cite{MJ1}.  There, sufficient conditions for nonlinear stability to $T$-periodic perturbations
were derived in terms of Jacobians of various maps between the parameter space $(a,E,c)$ and the
period, mass, and momentum.  As the theory for the gBBM equation is nearly identical to this work,
we only review the main points of the analysis here and refer the reader to \cite{MJ1} for details.

To begin, we assume the nonlinearity $f$ is such that the Cauchy problem for \eqref{gbbm} is globally
well posed on a real Hilbert space $X$ of $T$-periodic functions defined on $\RM$, which we equip
with the standard $L^2([0,T])$ inner product.  In particular, we require $L^2_{\rm per}([0,T])$
to be a subspace of $X$.  Also, we identify the dual space $X^*$ through the usual
pairing.  Now, we fix a periodic traveling wave solution of \eqref{gbbm} $u_0(x;a_0,E_0,c_0)$ and
notice that we can write the linearized spectral problem \eqref{linear} about $u_0$ in the form
\[
\mathcal{D}^{-1}\partial_x\mathcal{E}_0''(\phi)=\mu \phi
\]
where $\mathcal{E}_0(\phi)$ is an augmented energy functional defined by
\begin{align*}
\mathcal{E}_0(\phi)&=-\mathcal{E}(\phi)+c_0\mathcal{P}(\phi)+a_0\mathcal{M}(\phi)
\end{align*}
where $\mathcal{E}$, $\mathcal{P}$, and $\mathcal{M}$ are the energy, momentum, and energy functionals defined on $X$
defined by
\begin{align*}
\mathcal{E}(\phi)&=\int_0^T\left(\frac{1}{2}u(x)^2+F(u(x))\right)dx\\
\mathcal{M}(\phi)&=\int_0^Tu(x)dx\\
\mathcal{P}(\phi)&=\frac{1}{2}\int_0^T\left(\phi(x)^2+\phi_x(x)^2\right)dx.
\end{align*}
In particular, notice that each of these functionals are left invariant under spatial translations and hence
it is appropriate to study the stability of periodic solutions of \eqref{travelode} up to translation.
To this end, we introduce a semi-distance $\rho:X\times X\to\RM$ defined via
\[
\rho(\phi,\psi):=\inf_{\xi\in\RM}\|\phi(\cdot)-\psi(\cdot+\xi)\|_{X}
\]
and we seek conditions for which the following statement is true: If $\phi_0\in X$ is near $u_0$ as measured
by the semidistance $\rho$, then the solution $\phi(x,t)$ of \eqref{gbbm} with initial data $\phi_0$
stays close to a translate of $u_0$ for all time.
%
%

We begin by noticing that $u_0(x;a_0,E_0,c_0)$ is a critical point of the augmented energy functional
$\mathcal{E}_0$.  In order to determine the nature of this critical point, it is necessary to analyze
the second derivative of $\mathcal{E}_0$.  If $\mathcal{E}_0''(u_0)\in\mathcal{L}(X,X^*)$ is positive
definite, then nonlinear stability follows by standard arguments.  However, one sees after an easy calculation
that
\[
\mathcal{E}_0''(u_0)=\mathcal{L}
\]
which is clearly not positive definite by the translation invariance of \eqref{gbbm}.  Indeed, since $\mathcal{L}u_{0,x}=0$
and $u_0$ is not monotone it follows by standard Sturm-Liouville arguments that zero is either the second
or third eigenvalue of $\mathcal{L}$ with respect to the natural ordering on $\RM$.  By Lemma 4.1 of [MJ1],
we it follows that $\mathcal{L}$ considered on $L^2_{\rm per}([0,T])$
has precisely one negative eigenvalue, a simple eigenvalue at zero, and the rest
of the spectrum is positive and bounded away from zero if $T_E>0$.  If $T_E\leq 0$, then either the
the null space or the number or negative eigenvalues jumps by one: a situation which can seemingly
not be handled by the present variational analysis\footnote{However, see the recent work of Bronski, Johnson, and Kapitula [BrJK].}.

Thus, assuming $T_E>0$ it follows that $u_0$ is a degenerate critical point of $\mathcal{E}_0$ with one unstable
direction and one neutral direction.  In order to get rid of the unstable direction, simply notice that the evolution
of \eqref{gbbm} does not occur on the entire space $X$, but on the codimension two subset
\[
\Sigma_0:=\{\phi\in X:\mathcal{M}(\phi)=M(a_0,E_0,c_0),\;\;\mathcal{P}(\phi)=P(a_0,E_0,c_0)\}.
\]
Clearly $\Sigma_0$ is a smooth submanifold of $X$ containing all translates of the function $u_0$.
Defining $\mathcal{T}_0$ to be the tangent space of $\Sigma_0$ at $u_0$, i.e.
\[
\mathcal{T}_0:=\{\phi\in X:\left<u_0,\phi\right>=\left<1,\phi\right>=0\},
\]
we have by Lemma 4.3 of [MJ1] that the quadratic form induced by $\mathcal{E}_0''(u_0)$ is positive definite
on $\mathcal{T}_0$ if $T_E>0$ and
\begin{equation}
\left<\mathcal{E}_0''(u_0)\phi_2,\phi_2\right>=-\{T,M\}_{a,E}\{T,M,P\}_{a,E,c}<0\label{ip1}
\end{equation}
where $\phi_2$ is defined as above.
Since the underlying periodic wave is spectrally unstable if $\{T,M,P\}_{a,E,c}<0$ by Theorem \ref{thm:fw},
the only interesting case in which \eqref{ip1} holds is when $\{T,M\}_{a,E}$ and $\{T,M,P\}_{a,E,c}$ are positive.  With these conditions
in mind, it follows that the augmented energy $\mathcal{E}_0$ is coercive on $\sigma_0$ near $u_0$ with respect
to the semi-distance $\rho$, from which nonlinear stability follows: see Proposition 4.2 and the corresponding
proof in \cite{MJ1} for details.  Summarizing, we have the following orbital stability result.

\begin{thm}\label{orbital}
Let $u(x;a_0,E_0,c_0)$ be a periodic solution of \eqref{travelode} such that the Jacobians
$T_E$, $\{T,M\}_{a,E}$, and $\{T,M,P\}_{a,E,c}$ are positive at $(a_0,E_0,c_0)$.  Then there exists
positive constants $C_0$ and $\varepsilon_0$ such that if $\phi_0\in X$ satisfies $\rho(\phi_0,u)<\varepsilon$
for some $\varepsilon<\varepsilon_0$, then the solution $\phi(x,t)$ of \eqref{gbbm} with initial data $\phi_0$
satisfies $\rho(\phi(\cdot,t),u)\leq C_0\varepsilon$.
\end{thm}

By Theorem \ref{orbital}, it follows that $\{T,M,P\}_{a,E,c}>0$ may not be sufficient for nonlinear stability
of a periodic traveling wave of \eqref{gbbm}.  In particular, notice that in the solitary wave theory it is always
true that the operator $\mathcal{L}$ has only one $L^2(\RM)$ eigenvalue, and hence it is always true that any unstable
eigenvalues of the linearized operator must be real.  In the periodic context, however, we see this is only
true if $T_E>0$, which is not true for all periodic traveling wave solutions of \eqref{gbbm}: for example,
it is clear that the cnoidal wave solutions of the modified BBM equation corresponding to \eqref{gbbm} with
$f(u)=u^3$ of sufficiently long wave period satisfy $T_E<0$.  Thus, it may be possible in certain
situations that the linearized operator $\mathcal{A}$ has unstable $T$-periodic eigenvalues which are not real.
Moreover, even if $T_E$ and $\{T,M,P\}_{a,E,c}$ are positive, and hence one has periodic spectral stability, it is not
clear from this analysis whether one has nonlinear stability since the sign of the Jacobian $\{T,M\}_{a,E}$
still plays a seemingly large role\footnote{However, one should be aware of the recent works \cite{DK} and \cite{BrJK} in which
perturbation methods and more delicate functional analysis were used to analyze this problem for the gKdV
in a way which provides more precise results than the above variational methods.}.
This phenomenon, which stands in contrast to the solitary wave theory,
is a reflection of how the periodic traveling wave solutions of \eqref{gbbm} have a much richer structure
than the solitary waves, allowing for possibly more interesting dynamics.

\section{Modulational Instability Analysis}

In this section, we begin our study of the spectral stability of periodic traveling wave solutions
of the gBBM equation \eqref{gbbm} to arbitrary localized perturbations.  In particular, our methods
will detect instabilities of such solutions to long wavelength perturbations, i.e. to slow modulations
of the underlying wave.  Such stability analysis seems to be a bit more physical than the
periodic stability analysis conducted in the previous section in the following sense: in physical applications,
one should probably never expect to find an exact spatially periodic wave.  Instead, what one often sees is a
solution which on small space-time scales seems to exhibit periodic behavior, but over larger scales is clearly
seen not to be periodic due to a slow variations in the amplitude, frequency, etc, i.e. one sees
slow modulations physical parameters defining the solution.  Thus, it seems natural to study the stability of
such solutions by idealizing them as exact spatially periodic waves and then study
the stability of this idealized wave to slow modulations in the underlying parameters.
This is precisely the goal of the modulational stability analysis in this paper.
Moreover, in terms of the long-time stability of such solutions it is clear that a low frequency analysis of the
linearized operator is vital to obtaining suitable bounds on the corresponding solution operator of the nonlinear
equation.

To begin, notice that by Lemma \ref{Evans1} the linearized operator $\mathcal{A}$ has a $T$-periodic eigenvalue
at the origin in the spectral plane of multiplicity (generically) three.  Thus, as we allow small variations
in the Floquet parameter, we expect that there will be three branches of continuous spectrum which bifurcate
from the origin.  According to Proposition \ref{globalspec}, one of these branches must be confined to the
imaginary axis, and hence will not contribute to any spectral instability.  In order to determine
if the other two branches bifurcate off the imaginary axis or not, we derive an asymptotic
expansion of the function $D(\mu,e^{i\kappa})$ for $|\mu|+|\kappa|\ll 1$.  As a result, we will
see that, to leading order, the local structure of the spectrum near the origin is governed
by a homogeneous polynomial of degree three in the variables $\mu$ and $\kappa$.  We then evaluate
the coefficients of the resulting polynomial in terms of Jacobians of various maps from the
parameters $(a,E,c)$ to the quantities $T$, $M$, and $P$.  To this end, we begin with the following Lemma.

\begin{lem}\label{vanish}
If $\{T,M,P\}_{a,E,c}\neq 0$, the equation $D(\mu,e^{i\kappa})=0$ has the following normal
form in a neighborhood of $(\mu,\kappa)=(0,0)$:
\begin{equation}
-(i\kappa)^3+\frac{(i\kappa)^2\mu T}{c}+\frac{i\kappa\mu^2}{2}\left(\tr\left(\MM_{\mu\mu}(0)\right)-\frac{T^2}{c^2}\right)
             -\mu^3\{T,M,P\}_{a,E,c}+\mathcal{O}(|\mu|^4+\kappa^4)=0\label{evansnormalform}
\end{equation}
whose Newton diagram is depicted in Figure \ref{newton}.
\end{lem}

\begin{figure}
\begin{center}
\includegraphics[scale=.65]{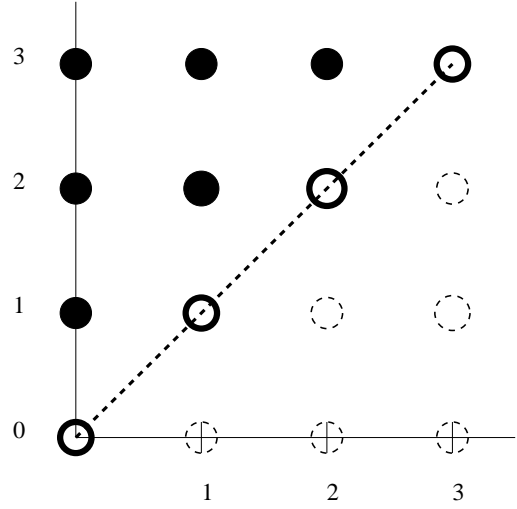}
\end{center}
 \caption{The Newton diagram corresponding to the asymptotic
expansion of $D(\mu,e^{i\kappa})=0$ in a neighborhood of
$(\mu,\kappa)=(0,0)$ is shown to $O(|\mu|^3+\kappa^3)$.  Terms associated
to open circles with dashed boundary are shown to vanish due to the natural symmetries
inherent in \eqref{linear}.  The open circles with dark boundary
are non-vanishing terms which are a part of the lower convex hull, and hence
contribute to the dominant balance.
The closed dark circles lie above the lower convex hull and
thus do not contribute to the leading order asymptotics.  See \cite{Baum} and \cite{BrJ} for more details.}
\label{newton}
\end{figure}

\begin{proof}
Define functions $a$ and $b$ on a neighborhood of $\mu=0$ by
\begin{equation}
D(\mu,e^{i\kappa})=
-\eta^3+(a(\mu)-3)\eta^2+b(\mu)\eta+D(\mu,1).\label{charpoly}
\end{equation}
where $\eta=e^{i\kappa}-1$ is small.  In particular, notice that
\begin{eqnarray*}
a(\mu)&=& \tr(\MM(\mu))\\ 
b(\mu)&=& \frac{1}{2}\left(\tr((\MM(\mu)-I)^2)-\tr(\MM(\mu)-I)^2\right).
\end{eqnarray*}
Now, using the fact that the spectral problem \eqref{linear} is invariant under the transformation $(x,\mu)\mapsto(-x,-\mu)$,
it follows that the matrices $\MM(-\mu)$ and $\MM(\mu)^{-1}$ are similar for all $\mu\in\CM$.  In particular,
it follows that
\begin{align*}
e^{-\mu T}\det(M(\mu)&-\lambda )=-\lambda^3 \det\left(M(-\mu)-\frac{1}{\lambda}\right)\\
&=-\lambda^3\left(\left(1-\frac{1}{\lambda}\right)^3+\left(a(-\mu)-3\right)\left(\frac{1}{\lambda}-1\right)^2
      +b(-\mu)\left(\frac{1}{\lambda}-1\right)+D(-\mu,1)\right)\\
&=-\left(\lambda-1\right)^3-\left(a(-\mu)-3\right)\lambda\left(\lambda-1\right)^2
        +b(-\mu)\lambda^2\left(\lambda-1\right)-\lambda^3D(-\mu,1)
\end{align*}
By comparing the $\mathcal{O}(\lambda^2)$ and $\mathcal{O}(\lambda^3)$ terms above to those in \eqref{charpoly},
we have the relations
\[
\left\{
  \begin{array}{ll}
    e^{-\mu T/c}a(\mu)=2a(-\mu)-b(-\mu)-3,\\
    -e^{-\mu T/c}=-a(-\mu)+b(-\mu)-D(-\mu,1)+2.
  \end{array}
\right.
\]
Differentiating with respect to $\mu$ and evaluating at $\mu=0$ immediately implies $b'(0)=0$ and $a'(0)=\frac{T}{c}$.
Similarly, it follows that $b''(0)=a''(0)-\frac{T^2}{c^2}=\tr\left(\MM_{\mu\mu}(0)\right)-\frac{T^2}{c^2}$.  The proof is now
complete by Lemma \ref{Evans1} and the fact that $\eta=i\kappa+\mathcal{O}(\kappa^2)$.
\end{proof}

It follows that the structure of $\spec\left(\mathcal{A}\right)$ in a neighborhood of the origin is, to leading order,
determined by the above homogeneous polynomial in $\kappa$ and $\mu$.  Due to the triple root of $D(\cdot,1)$ at $\mu=0$
the implicit function theorem fails, but can be trivially corrected by considering the appropriate
change of variables.  
This leads us to the following theorem giving a modulational stability index for traveling wave solutions of \eqref{gbbm}.

\begin{thm}\label{thm:modstable}
With the above notation, define
\begin{align*}
%
%
%
\Delta&=
\frac{1}{4}\left(\tr(\MM_{\mu\mu}(0))-\left(\frac{T}{c}\right)^2\right)^2\left(2\tr(\MM_{\mu\mu}(0))-\left(\frac{T}{c}\right)^2\right)\\
      &-27\{T,M,P\}_{a,E,c}^2
      +6\{T,M,P\}_{a,E,c}\left(\frac{3\tr(\MM_{\mu\mu}(0))}{2}-\frac{5\left(\frac{T}{c}\right)^2}{3}\right)\left(\frac{T}{c}\right)
\end{align*}
and suppose that $\{T,M,P\}_{a,E,c}\neq 0$.
If $\Delta>0$, then the spectrum of the linearized operator $\mathcal{A}$ in a neighborhood of the origin
consists of the imaginary axis with a triple covering.  If $\Delta<0$, then $\sigma(\mathcal{A})$ in a neighborhood
of the origin consists of the imaginary axis with multiplicity one together with two curves which are tangent to lines through the origin.
\end{thm}

\begin{proof}
Since, to leading order in $\mu$ and $\kappa$, the Evans function is homogeneous by Lemma \ref{vanish} it seems
natural to work with the projective coordinate $y=\frac{i\mu}{\kappa}$.  Making such a change of variables,
Lemma \ref{vanish} implies the equation $D(\mu,e^{i\kappa})=0$ can be written as
\begin{equation}
1+\frac{yT}{c}-\frac{y^2}{2}\left(\tr\left(\MM_{\mu\mu}(0)\right)-\left(\frac{T}{c}\right)^2\right)-y^3\{T,M,P\}_{a,E,c}+
\kappa E(\kappa,y)=0\label{evansprojective}
\end{equation}
where $E(\kappa,y)$ is continuous in a neighborhood of the origin.  Let $y_{1,2,3}$ denote the three roots
of the above cubic in $y$ corresponding to $E(\kappa,y)=0$.  Assuming $\Delta\neq 0$ it follows
that $y_{1,2,3}$ are distinct and hence the implicit function theorem applies giving three distinct
solutions of \eqref{evansprojective} in a neighborhood of each of the $y_{1,2,3}$.  In terms of the
original variable $\mu$, this gives three solution branches
\[
\mu_{1,2,3}=-iy_{1,2,3}\kappa+\mathcal{O}(\kappa^2).
\]
If $\Delta>0$, then $y_{1,2,3}\in\RM$, giving three branches of spectrum emerging from the origin tangent to the imaginary axis.
From the Hamiltonian symmetry of \eqref{linear}, the spectrum is symmetric with respect to reflections across the imaginary
axis and hence $\Delta>0$ implies these thee branches of spectrum must in fact lie on the imaginary axis, proving the existence
of an interval of spectrum of multiplicity three on the imaginary axis.  In the case $\Delta<0$, it follows that one of the roots,
$y_1$ say, is real while the other two $y_{2,3}$ occur in a complex conjugate pairs, giving one branch
along the imaginary axis and two branches emerging from the origin tangent to lines through the origin with angel $\arg(-iy_{2,3})$:
see Figure \ref{localnormalform}.
\end{proof}

\begin{figure}
\begin{center}
\includegraphics[scale=.35]{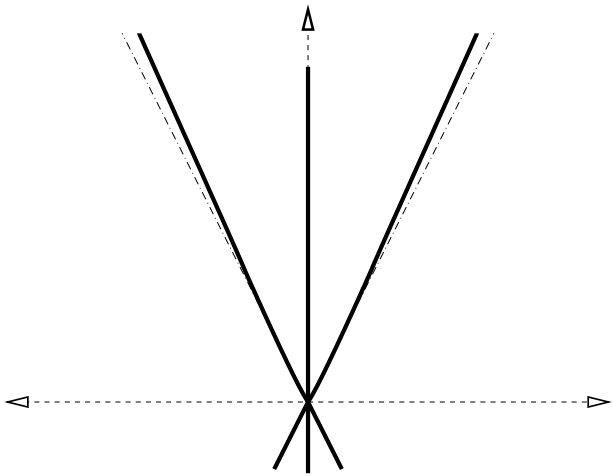}
\end{center}
 \caption{When $\Delta<0$, the local normal form of
$\spec(\mathcal{A})$ consists of a segment of the imaginary
axis union with two straight lines making equal angles with the
imaginary axis.  Notice that these lines intersect at the origin,
corresponding to the fact that $D(\mu,1)=\mathcal{O}(|\mu|^3)$.
In this picture, the horizontal dashed line represents the real axis,
and the dashed lines which are tangent to the straight dark lines represent the
true spectrum of the linearized operator.}\label{localnormalform}
\end{figure}

\begin{rem}
The modulational instability index $\Delta$ derived above is considerably more complicated than the one
derived by Bronski and Johnson \cite{BrJ} for generalized Korteweg-de Vries equation
\begin{equation}
u_t=u_{xxx}-cu_{x}+\left(f(u)\right)_x\label{gkdv}.
\end{equation}
When considering the spectral stability of periodic traveling wave solutions of \eqref{gkdv}, it was shown
that there exists a modulational instability index $\Delta_{gKdV}$ such that $\Delta_{gKdV}<0$ implies
modulational instability, and $\Delta_{gKdV}>0$ implies modulational stability.  In this case,
the dominant balance is somewhat simpler due to the fact that the trace of the operator $\MM_{\mu}(0)$
vanishes, and hence the $\kappa^2\mu$ term in the corresponding Newton diagram vanishes.
As a result, the modulational instability index took the form
\[
\Delta_{gKdV}=\frac{1}{2}\left(\tr(\MM_{\mu\mu}(0))\right)^3-\frac{1}{3}\left(\frac{3}{2}\{T,M,P\}_{a,E,c})\right)^2,
\]
where here $D$ represents the Evans function for equation \eqref{gkdv}.
It follows that the sign of $D_{\mu\mu\mu}(0,1)$ does not effect the modulational instability of such
periodic solutions of equation \eqref{gkdv}.  However, in the case of the gBBM equation, the fact
that $a'(0)\neq 0$ seems to suggest that the sign of the sign of this orientation index (see next section)
has an impact on the modulational stability of the periodic traveling wave.  In fact, in section 4 we prove
exactly this fact in the case of a power-nonlinearity $f(u)=u^{p+1}/(p+1)$: we prove the long-wavelegth
periodic traveling wave solutions of \eqref{gbbm} are modulationally unstable if and only if the limiting
homoclinic orbit (solitary wave) is exponentially unstable.
\end{rem}

Our next goal is to use the integrable nature of \eqref{travelode} to express $\tr(\MM_{\mu\mu}(0))$
in terms of the underlying periodic traveling wave $u$.  This is the content of the following lemma.

%

\begin{lem}\label{EvalIndex}
We have the following identity:
\begin{align*}
%
%
\frac{1}{2}\tr\left(\MM_{\mu\mu}(0)\right)&=\{T,P\}_{E,c}+2\{M,P\}_{a,E}-V'(u_{-})\{T,M\}_{a,E}.
\end{align*}
\end{lem}


\begin{proof}
This proof is essentially an extension of that of Lemma \ref{Evans1}.  Using the same notation, a straightforward
yet tedious calculation yields
\begin{align*}
\tr(\MM_{\mu\mu}(0))&=-\frac{2}{\mu^2}\tr\left({\rm cof}\left(\delta\WM(\mu)\WM(0,0)^{-1}\right)\right)\big{|}_{\mu=0}\\
&=2\{T,P\}_{E,c}+2\{M,P\}_{a,E}+2\{T,M\}_{a,c}-\frac{2}{c}V'(u_{-})\{T,M\}_{a,E}
\end{align*}
The proof is completed by noting that $\{T,M\}_{a,c}=\{M,P\}_{E,a}$ since $T_c=P_E$ and $M_c=P_a$.%
\end{proof}

Thus, the modulational stability of a given periodic traveling wave solution to \eqref{gbbm} can be determined
from information about the underlying solution itself.  Interestingly, from Lemma \ref{EvalIndex} the modulational
instability index seems to depend directly on the turning point $u_{-}$, a feature not seen in the corresponding
analysis of the generalized KdV equation in \cite{BrJ}.  However, one can trace this dependence back to the definition
of $w_1^1$ in Lemma \ref{Evans1} and hence seems to be unavoidable.  In the next section, we will analyze the
above stability indices in neighborhoods of homoclinic orbits in phase space, complimenting the results of
Gardner \cite{Gard} by providing stability results in this limit.

\section{Analysis of Stability Indices in the Solitary Wave Limit}

The goal of this section is to study the long wavelength asymptotics of the stability indices
derived in the previous sectios.  Throughout, we restrict ourselves to the case of a power-nonlinearity
$f(u)=u^{p+1}/(p+1)$: this restriction is vital to our calculation since in this case we gain an additional
scaling symmetry.  In particular, if $v(x;a,E)$ satisfies the differential equation
\begin{equation}
\frac{1}{2}v_x^2-v^2+\frac{1}{(p+1)(p+2)}v^{p+2}=au+E,\label{eqn:v}
\end{equation}
then a straight forward calculation shows that we can express the periodic solution $u(x;a,E,c)$ as
\begin{equation}
u(x;a,E,c)=(c-1)^{1/p}v\left(\left(\frac{c-1}{c}\right)^{1/2}x;\frac{a}{c^{1+1/p}},\frac{E}{c^{1+2/p}}\right).\label{scaling:gbbm}
\end{equation}
This additional scaling allows explicit calculations of $P_c$, which ends up determining the stability
of periodic traveling wave solutions of \eqref{gbbm} of sufficiently long wavelength.

A reasonable guess would be that long-wavelength periodic traveling
wave solutions of \eqref{gbbm} have the same stability properties as the limiting homoclinic
orbit (solitary wave).  However, as noted in the introduction, this is a highly singular limit
and so it is not immediately clear whether such results are true.  It is well known that the solitary wave
is spectrally unstable if and only if $p>4$ and $1<c<c_0(p)$ for some critical wave speed $c_0(p)$.
It follows from the work of Gardner \cite{Gard} 
that periodic waves below the homoclinic orbit in phase space
which are sufficiently close to the homoclinic orbit are unstable if the solitary wave
is unstable.  In particular, it is proved that the linearized operator $\mathcal{A}$ for the
periodic traveling wave $u$ with sufficiently long wavelength has a ``loop" of spectrum
in the neighborhood of any unstable eigenvalues of the limiting solitary wave.  Thus, Gardner's
analysis deduces instability of long wavelength periodic waves from instability of the limiting
solitary waves.  The results of this section compliment this theory by also proving that the
stability of the limiting wave is inherited by nearby periodic waves below the homoclinic orbit
in phase space.

In terms of the finite-wavelength instability index, it seems reasonable by
Theorem \ref{thm:fw} to expect that for periodic traveling wave solutions below the seperatrix
of sufficiently long wavelength, $\{T,M,P\}_{a,E,c}<0$ for if and only if $p>4$ and $1<c<c_0(p)$.
What is unclear is whether such a result should be true for the modulational instability
index $\Delta$.  Indeed, although Gardner's results prove that the spectrum of the linearization
about a periodic traveling wave of sufficiently long wavelength in the neighborhood of the origin
contains the image of a continuous map of the unit circle, to our knowledge it has never
been proved that this map is injective.  Thus, it is not clear from Gardner's results whether
a modulational instability will arise from this eigenvalue since it is possible this ``loop" is confined
to the imaginary axis: we show that in fact one has modulational instability in this limit precisely
when the limiting solitary wave is unstable.

The main result for this section is the following theorem, which is
%
based on asymptotic estimates
of the instability indices derived in section 3.  In particular,
we prove the sign of both instability indices in the solitary wave limit is determined by the sign
of $\frac{\partial}{\partial c}P$, where $P=P(a,E,c)$ is the momentum of the periodic wave $u(x;a,E,c)$.
The proof is based on a more technical lemma, which shows that 
$\frac{\partial}{\partial a}M(a,E,c)<0$ for waves of sufficiently long wavelength, i.e. for $a,E$ sufficiently
close to zero.  We begin by outlining the proof of the following theorem, and then fill in the necessary lemma's
afterward.

\begin{thm}\label{solitarylim}
Let $f(u)=u^{p+1}$ for some $p\geq 1$ and let $u(x;a,E,c)$, with $a$ sufficiently small, be a periodic solution of the traveling
wave ODE \eqref{travelode} which corresponds to an orbit below the homoclinic orbit in phase space.
If $T(a,E,c)$ is sufficiently large, then $u$ is modulationally and nonlinearly stable for all $1\leq p<4$.
For $p>4$, there exists a critical wave speed $c(p)>1$ such that if $T(a,E,c)$ is sufficiently large,
the solution $u$ is orbitally and modulationally stable if $c>c(p)$, while it is
modulationally unstable for $1<c<c(p)$.
\end{thm}

%
%

\begin{proof}
Without loss of generality, suppose that the periodic solution $u(x;a,E,c)$ of \eqref{travelode}
corresponds to a branch cut of the function $\sqrt{E-V(u;a,c)}$ with positive right end point.
Now, when $a$ and $E$ are small there are two turning points $r_1,r_2$ in the
neighborhood of the origin and a third turning point $r_3$ which
is bounded away from the origin.
In the solitary wave limit $a,E\to 0$ a straight forward calculation
gives that $r_2-r_1=\mathcal{O}\left(\sqrt{a^2-2(c-1)E}\right)$
From this, it follows that the period satisfies the asymptotic relation
\begin{equation}
T(a,E,c)=\mathcal{O}\left(\ln\left(a^2-2(c-1)E\right)\right).\label{largeperiod}
\end{equation}
To see this, notice that by our assumptions on $r_1$, $r_2$, and $r_3$
we can write $E-V(u;a,1)=(r-r_1)(r-r_2)(r_3-r)Q(r)$ where $Q(r)$ is positive on the set $[r_1,r_3]$.
The period can then be expressed as
\begin{align*}
T(a,E,c)&=\sqrt{2c}\int_{r_2}^{r_3}\frac{dr}{\sqrt{(r-r_1)(r-r_2)(r_3-r)Q(r)}}\\
&=\sqrt{2c}\int_{r_2-r_1}^{r_2-r_1+\delta}\frac{dr}{\sqrt{r(r-(r_2-r_1))(r_3-r_1-r)Q(r+r_1)}}\\
&+\sqrt{2c}\int_{r_2-r_1+\delta}^{r_3-r_1}\frac{dr}{\sqrt{r(r-(r_2-r_1))(r_3-r_1-r)Q(r+r_1)}}
\end{align*}
The integral over the set $(r_2-r_1+\delta,r_3-r_1)$ is clearly $\mathcal{O}(1)$ as $(a,E)\to(0,0)$.
For the other integral, notice that
\[
(r_3-r_1-r)Q(r+r_1)>0
\]
on the set $[r_2-r_1,r_2-r_1+\delta]$.  Thus, in the limit as $(a,E)\to\infty$ we have
\begin{align*}
T(a,E,c)&\sim\int_{r_2-r_1}^{r_2-r_1+\delta}\frac{dr}{\sqrt{r(r-(r_2-r_1))}}\\
&=-4\ln(4(r_2-r_1))+2\ln(2(\sqrt{r_2-r_1+\delta}+\sqrt{\delta}))\\
&\sim-\ln(r_2-r_1)
\end{align*}
from which \eqref{largeperiod} follows.  Similar computations yield the following asymptotic relations
for $a$ and $E$ sufficiently small:
\begin{align*}
P(a,E,c)&=\mathcal{O}(1)\\
M(a,E,c)&=\mathcal{O}\left(a\ln\left(a^2-2(c-1)E\right)\right)\\
M_a(a,E,c)&=\mathcal{O}\left(\frac{a^2}{a^2-2(c-1)E}\right)\\
T_a(a,E,c)&=\mathcal{O}\left(\frac{a}{a^2-2(c-1)E}\right)=M_E(a,E,c)\\
T_E(a,E,c)&=\mathcal{O}\left(\frac{1}{a^2-2(c-1)E}\right)\\
T_c(a,E,c)&=\mathcal{O}\left(\frac{E}{a^2-2(c-1)E}\right)=2P_E+\frac{1}{c}T.
\end{align*}
Thanks to the  above scaling we know $M_c=2P_a+\frac{1}{c}T$ can be expressed as a linear combination
of $M$, $M_a$, and $M_E=T_a$.  Similarly, $P_c$ can be expressed as a linear combination of $P$, $2P_a=M_c-\frac{1}{c}T$ and $P_E$.
It follows that the asymptotically largest minor of $\{T,M,P\}_{a,E,c}$ for $a$ and $E$ small is $-T_E M_a P_c$ and, moreover,
\[
\tr\left(\MM_{\mu\mu}(0)\right)\sim T_E P_c.
\]
In Lemma \ref{lem:mass} below, we will show that $M_a<0$ for such $a$ and $E$.  Moreover, it is clear that for such
periodic waves with sufficiently long wavelength satisfy $T_E>0$.  Therefore, it follows
%
%
%
that both stability indices are determined by the sign of $P_c(a,E,c)$ in the solitary wave limit.
The theorem now follows by Lemma \ref{lem:momentum} below.
\end{proof}

\begin{rem}
Notice that the fact that the finite wavelength instability index is determined by the sign of $P_c$ in the solitary
wave limit is not surprising, since this is exactly what detects the stability of the limiting solitary waves.  What is
surprising is that the same quantity controls the modulational stability index in the same limit.  As mentioned
above, it has not been know if the instability of the limiting solitary wave forces a modulational instability:
the answer is shown to be affirmative by Theorem \ref{solitarylim}.
\end{rem}

In order to complete the proof Theorem \ref{solitarylim}, we must prove a few more technical lemmas.  The first
is used in showing that the sign of the modulational and finite-wavelength instability indices
are determined completely by the sign of $P_c(a,E,c)$ in the limit as $a,E$ tend to zero.  This is the
result of the following lemma.

\begin{lem}\label{lem:mass}
For $a,E$ sufficiently small, $M_a(a,E,c)\leq 0$ for all $c>1$.
\end{lem}

\begin{proof}
Notice it is sufficient to prove $\frac{\partial}{\partial a}M(a,0,c)\leq 0$ for all $c>1$ and $a$ sufficiently
small.  Now, $M(a,0,c)$ can be written as
\[
M(a,0,c)=\sqrt{2c} \int_0^{r(a,c)}\frac{\sqrt{u}\;du}{\sqrt{a+\frac{c-1}{2}u-\frac{1}{(p+1)(p+2)}u^{p+1}}}
\]
where $r(a,c)$ is the smallest positive root of the polynomial equation $a+\frac{c-1}{2}r-\frac{1}{(p+1)(p+2)}r^{p+1}=0$.
Setting $a=((p+1)(p+2))^{1/p}\alpha$, we have
\[
M\left(\alpha,E,c\right)=((p+1)(p+2))^{1/p}\sqrt{2c}\int_0^{\widetilde{r}(\alpha,c)}\frac{\sqrt{u}\;du}{\sqrt{\alpha+\frac{c-1}{2}u-u^{p+1}}}
\]
where $\widetilde{r}(\alpha,c)$ is the smallest positive root of the polynomial $\alpha+\frac{c-1}{2}r-r^{p+1}=0$.
Notice for a fixed wave speed $c>1$, $\widetilde{r}(\alpha,c)$ is a smooth function of $\alpha$ for $\alpha$
sufficiently small and satisfies
\[
\widetilde{r}(\alpha,c)=\left(\frac{c-1}{2}\right)^{1/p}+\frac{2a}{p(c-1)}+\mathcal{O}(a^2).
\]
The goal is to now rewrite the above integral over a fixed domain and show that the integrand is
a decreasing function of $\alpha$ for a fixed $c>1$.

Making the substitution $u\to \widetilde{r}(\alpha,c)u$ yields the expression
\[
\frac{M(\alpha,0,c)}{((p+1)(p+2))^{1/p}\sqrt{2c}}=\int_0^1\frac{\sqrt{u}\;du}{\sqrt{\alpha\;\widetilde{r}(\alpha,c)^{-3}
         +\left(\frac{c-1}{2}\right)u\;\widetilde{r}(\alpha,c)^{-2}
                      -\widetilde{r}(\alpha,c)^{p-2}\;u^{p+1}}}
\]
Now, we can use the above expansion of $r(a,c)$ to conclude that
\begin{align*}
\frac{\partial}{\partial \alpha}\big{|}_{\alpha=0}\left(\alpha\;\widetilde{r}(\alpha,c)^{-3}
         +\left(\frac{c-1}{2}\right)u\;\widetilde{r}(\alpha,c)^{-2}
                      -\widetilde{r}(\alpha,c)^{p-2}\;u^{p+1}\right)
\end{align*}
is positive on the open interval $(0,1)$ for all $p\geq 1$ and $c>1$, which completes the proof.
\end{proof}

Finally, it is left to analyze the %
%
asymptotic behavior for a fixed
wavespeed $c>1$ of the quantity $P_c(a,E,c)$ in the solitary wave limit.  This is the content of the following
lemma.

\begin{lem}\label{lem:momentum}
In the case of power non-linearity $f(u)=u^{p+1}/(p+1)$, we the momentum $P=P(a,E,c)$ satisfies
\[
\frac{\partial }{\partial c} P\left(a,E,c\right)=\frac{(c-1)^{2/p-1/2}c^{1/2}I\left(\frac{4}{p}\right)}{2pc(c-1)}
                        \left(4c-p+\frac{(4c+p)(c-1)p}{(4+p)c}\right) + \mathcal{O}(|a|+|E|)
\]
in the solitary wave limit $(a,E)\to(0,0)$, where $I(r)=\int_{-\infty}^{\infty}\sech^r(x)dx$.  In particular, for $a$
and $E$ sufficiently small, if $p<4$ then
$\frac{\partial }{\partial c} P\left(a,E,c\right)>0$ for all $c>1$ while if $p>4$ then
$\frac{\partial }{\partial c} P\left(a,E,c\right)<0$ for $1<c<c_0(p)$ and $\frac{\partial }{\partial c} P\left(u;a,E,c\right)>0$
for $c>c_0(p)$, where
\[
c_0(p)=\frac{p\left(1+\sqrt{2+\frac{1}{2}p}\;\right)}{4+2p}.
\]
\end{lem}

\begin{proof}
The proof is based on scaling and a limiting argument, as well as a modification of the analysis
in \cite{PW1}.  To begin, let $v=v(x;a,E)$ satisfy the differential equation
\eqref{eqn:v} so that $u(x;a,E,c)$ can be expressed via scaling as in \eqref{scaling:gbbm}, and assume with out loss of generality
that $x=0$ be an absolute max of $v(x;a,E,c)$.
Clearly, the solitary wave limit corresponds to taking $(a,E)\to(0,0)$ with fixed wave speed $c>1$.  Notice that
on any compact subset $\Gamma$ of $\RM$, we have
\[
v(x;a,E)\to \left(\frac{(p+2)(p+1)}{2}\right)^{1/p}\sech^{2/p}\left(\frac{p}{2}x+x_0\right)
\]
uniformly as $(a,E)\to(0,0)$ on $\Gamma$ for some $x_0\in \RM$.  Using \eqref{eqn:v}, it follows that
\[
\int_0^{T((c-1)/c)^{1/2}}v^2(x)dx=\sqrt{2}\int_{\tilde{u}_{-}}^{\tilde{u}_{+}}\frac{u^2\;du}{\sqrt{E+v^2-\frac{1}{(p+1)(p+2)}v^{p+2}+av}}
\]
where $\tilde{u}_{\pm}$ are the roots of $E+v^2-\frac{1}{(p+1)(p+2)}v^{p+2}+av=0$ satisfying the original
hypothesis of the roots $u_{\pm}$ of $E-V(u;a,c)=0$.
Since $\int_{-T((c-1)/c)^{1/2}/2}^{T((c-1)/c)^{1/2}/2} v^2(x)dx=\mathcal{O}(1)$ as $(a,E)\to(0,0)$, the dominated convergence theorem along with the fact
that $\int_{-T((c-1)/c)^{1/2}/2}^{T((c-1)/c)^{1/2}/2}v^2(x)dx$ is a $C^1$ function of $a$ and $E$ implies that
\begin{eqnarray*}
\int_{-T((c-1)/c)^{1/2}/2}^{T((c-1)/c)^{1/2}/2}v(x;a,E)^2dx=\left(\frac{(p+2)(p+1)}{2}\right)^{2/p}\frac{1}{p}I\left(\frac{4}{p}\right)+
                     \mathcal{O}(|a|+|E|).
\end{eqnarray*}
Similarly, it follows that
\[
\int_{-T((c-1)/c)^{1/2}/2}^{T((c-1)/c)^{1/2}/2}v_x(x;a,E)^2dx=\left(\frac{(p+2)(p+1)}{2}\right)^{2/p}
                                  \frac{1}{4+p}\;I\left(\frac{4}{p}\right)+\mathcal{O}(|a|+|E|).
\]
Using \eqref{scaling:gbbm}, we now have
\begin{align*}
p\left(\frac{(p+2)(p+1)}{2}\right)^{-2/p}&\int_{-T/2}^{T/2}\left(u(x;a,E,c)^2+u_x(x;a,E,c)^2\right)dx\\
&=(c-1)^{2/p-1/2}c^{1/2}I\left(\frac{4}{p}\right)
+(c-1)^{2/p+1/2}c^{-1/2}\frac{p}{4+p}\;I\left(\frac{4}{p}\right)\\
&+\mathcal{O}(|a|+|E|)
\end{align*}
as $a,E\to 0$, and hence it follows by differentiation that
\begin{align*}
p\left(\frac{(p+2)(p+1)}{2}\right)^{-2/p}&\frac{\partial}{\partial c}\int_{-T/2}^{T/2}\left(u(x;a,E,c)^2+u_x(x;a,E,c)^2\right)dx=\\
          &\frac{(c-1)^{2/p-1/2}c^{1/2}I\left(\frac{4}{p}\right)}{2pc(c-1)}
                        \left(4c-p+\frac{(4c+p)(c-1)p}{(4+p)c}\right)\\
&+\mathcal{O}(|a|+|E|)
\end{align*}
as claimed, where we have used that $T_c u_{-}^2=\mathcal{O}(|a|+|E|)$.
The lemma now follows by solving the quadratic equation $(4+p)c(4c-p)+(4c+p)(c-1)p=0$ for $c$
and recalling the restriction that $c>1$.
\end{proof}

The proof of Theorem \ref{solitarylim} is now complete by Lemmas \ref{lem:mass} and \ref{lem:momentum}.
As a consequence, the finite-wavelength instability index $\{T,M,P\}_{a,E,c}$ seems to be
a somewhat natural generalization of the solitary wave stability index, in the sense that the well known stability properties of solitary waves
are recovered in a long-wavelength limit.  Moreover, this gives an extension of the results of Gardner
in the case of generalized BBM equation by proving the marginally stable eigenvalue of the solitary
wave at the origin contributes to modulational instabilities of nearby periodic waves when ever the solitary
wave is unstable.

Finally, we wish to make an interesting comment concerning the elliptic function solutions of the modified BBM (mBBM) equation
with $f(u)=u^3$.  In the special case in our theory of $a=0$, it turns out that all periodic traveling wave solutions
$u(x;0,E,c)$ can be expressed in terms of Jacobi elliptic functions: when $E>0$ such solutions correspond to
cnoidal waves, expressible in terms of the $\cn(x,k)$ function, while when $E<0$ such solutions correspond
to dnoidal waves, expressible in terms of the function $\dn(x,k)$.  By Theorem \ref{solitarylim}, it follows that
the dnoidal waves with $1-k$ sufficiently small(corresponding to $0<-E\ll 1$) are always
modulationally stable and nonlinearly (orbitally) stable to perturbations with the same periodic structure.
What is unclear from our above analysis is whether the cnoidal waves with $k$ sufficiently close to $1$ (corresponding
to $0<E\ll 1$) are exhibit a sense of stability: such waves correspond to periodic orbits outside the seperatrix in phase
space, and hence Theorem \ref{solitarylim} does not apply to such solutions.  Moreover, such waves do not
approach any particular solitary wave in any uniform way, even on compact subsets, and hence
it is not clear if the stability of the limiting solitary waves at all influences the cnoidal waves.
However, notice that for $k$ sufficiently
close to unity one must have $T_E<0$ and hence by the analysis in the proof of Theorem \ref{solitarylim} we have $\Delta<0$.
Therefore, the cnoidal waves of the mBBM with sufficiently long wavelength are modulationally unstable.
Similarly, in this solitary wave limit one has
\[
\lim_{E\to 0^+}{\rm sign}\left(\{T,M,P\}_{a,E,c}(0,E,c)\right)=\lim_{E\to 0^+}{\rm sign}\left(M_a(0,E,c)\right)
\]
and hence one has spectral instability to periodic perturbations if $M_a(0,E,c)<0$ for $E>0$ sufficiently small.
Using the complex analytic calculations of Bronski, Johnson, and Kapitula \cite{BrJK}, we can solve the
corresponding Picard-Fuchs system and see that
\begin{equation}
\lim_{E\to 0^+}\frac{{\rm disc}(R(u;0,E,c))M_a(0,E,c)}{(c-1)T(0,E,c)}=1.\label{MassSolitary}
\end{equation}
See the appendix for more details of this calculation.
Since ${\rm disc}(R(u;0,E,c))<0$ for such solutions, we find that $M_a<0$.  Thus, the cnoidal wave solutions
of the mBBM equation of sufficiently large wavelength are exponentially unstable to perturbations with the same
periodic structure.  In particular, it seems like the index ${\rm sign}(T_E)$ serves as Maslov index in this
problem: when $T_E>0$ it is $P_c<0$ which signals the stability near the solitary wave, while it is $P_c>0$
when\footnote{While we have not actually proven this here (we only showed this is the case for cnoidal solutions of mBBM), we suspect
this is the case for all nonlinearities of the form $f(u)=u^{2k+1}$ for $k\in\mathbb{N}$.  For examples where this phenomenon
is studied in more cases see \cite{BrJK}.}  $T_E<0$.
This index arises naturally in the solitary wave setting: see \cite{Bridges:Maslov} for details
and discussion.  In particular, notice that in the solitary wave setting of Pego and Weinstein and in the
long wavelength analysis of Gardner, the Maslov index always has the same sign: however, as seen here this
is not the case when considering the full family of periodic traveling waves of the gBBM.

\section{Concluding Remarks}

In this paper, we considered the stability of periodic traveling wave solutions of the generalized
Benjamin-Bona-Mahony equation with respect to periodic and localized perturbations.  Two stability indices
were introduced of the full four parameter family of periodic traveling waves.
The first, which is given by the Jacobian of the map
between the constants of integration of the traveling wave ordinary
differential equation and the conserved quantities of the partial differential
equation restricted to the class of periodic traveling wave solutions,
serves to count (modulo $2$) the number of periodic
eigenvalues along the real axis. This is, in some sense, a natural
generalization of the analogous calculation for the solitary wave solutions,
and reduces to this is the solitary wave limit. The second index, which arises
as the discriminant of a cubic which governs the normal form of the
linearized operator in a neighborhood of the origin, can also
be expressed in terms of the conserved quantities of the partial differential
equation and their derivatives with respect to the constants of integration
of the ordinary differential equation. This discriminant detects modulational
instabilities of the underlying periodic wave, i.e. bands of spectrum off
of the imaginary axis in an arbitrary neighborhood of the origin.

In our calculations, we heavily used the fact that the ordinary differential equation defining the
traveling wave solutions of \eqref{gbbm} have sufficient first integrals, and as such is doubtlessly
related to the multi-symplectic formalism of Bridges \cite{B}.  Many of the ideas from this paper were recently
used by the author, in collaboration with Jared C. Bronski, in the analogous study of the spectral
stability of periodic traveling wave solutions of the generalized Korteweg-de Vries equation \cite{BrJ}.  Thus,
the general technique of this paper has been successful in two different cases.  These successes
gives an indication that a rather general modulational stability theory can be developed using
these techniques.  It would be quite interesting to compare the rigorous results from this paper to the formal
modulational stability predictions of Whitham's modulation theory \cite{W1}.  Moreover, the techniques in this paper
could eventually lead to a rigorous justification of the Whitham modulation equations for such dispersive equations.
%

It is not known if our techniques can be extended to equations which are deficient in the number
of first integrals.  An example would be the generalized regularized Boussinesq equation
\[
u_{tt}-u_{xx}-\left(f(u)\right)_{xx}-u_{xxtt}=0.
\]
A straight forward calculation shows that although this equation admits a five parameter family of traveling wave solutions
with a {\it four} parameter submanifold of periodic solutions.  Although one can still build a four dimensional
basis of null-space of the corresponding linearized operator, it is not clear whether the resulting
stability indices can be directly related back to geometric information concerning the original
periodic traveling wave.
Such a result would be very interesting, as it would allow the use of such techniques to several more classes of partial
differential equations admitting traveling wave solutions.

{\bf Acknowledgements:} The author would like to thank Jared C. Bronski for his many useful
conversations in the early stages of this work.  Also, the author gratefully acknowledges
support from a National Science Foundation Postdoctoral Fellowship under
grant DMS-0902192.

\section{Appendix}

In this appendix, we outline the general complex analytic methods used to derive the identity
in equation \eqref{MassSolitary}.  To begin, consider \eqref{gbbm} with $f(u)=u^3$
and let $u(x;a,E,c)$ be a given periodic traveling wave.  To begin, we notice that
that in this case
%
$T$, $M$, and $P$ can be
represented integrals on a Riemann surface corresponding to the classically allowed region for
$\sqrt{E-V(u;a,c)}$, where $V(u;a,c)=E+au+\left(\frac{c-1}{2}\right)u^2-u^4/4$.
In particular, defining $R(u;a,E,c):=E-V(u;a,c)$ we have the following integral representations:
\begin{align*}
T(a,E,c)&=\sqrt{\frac{c}{2}}\oint_{\Gamma}\frac{du}{\sqrt{R(u)}}\\
M(a,E,c)&=\sqrt{\frac{c}{2}}\oint_{\Gamma}\frac{u~du}{\sqrt{R(u)}}\\
P(a,E,c)&=\frac{1}{2}\left(\sqrt{\frac{c}{2}}\oint_{\Gamma}\frac{u^2~du}{\sqrt{R(u)}}+K\right)
\end{align*}
where $K$ is the classical action defined in \eqref{action}
and the integrals are taken around the classically allowed region corresponding to
the periodic orbit $u(x;a,E,c)$.
By differentiating the above relations with respect to the traveling wave parameters, we
have the following identities:
\begin{align*}
\nabla_{a,E,c}T(a,E,c)&=\left<-\frac{1}{2}\sqrt{\frac{c}{2}}\oint_{\Gamma}\frac{u~du}{R(u)^{3/2}},~
-\frac{1}{2}\sqrt{\frac{c}{2}}\oint_{\Gamma}\frac{du}{R(u)^{3/2}},~ -\frac{1}{4}\sqrt{\frac{c}{2}}\oint_{\Gamma}\frac{u^2~du}{R(u)^{3/2}}+\frac{1}{2c}T\right>\\
\nabla_{a,E,c}M(a,E,c)&=\left<-\frac{1}{2}\sqrt{\frac{c}{2}}\oint_{\Gamma}\frac{u^2~du}{R(u)^{3/2}},~
T_a,~
 -\frac{1}{4}\sqrt{\frac{c}{2}}\oint_{\Gamma}\frac{u^3~du}{R(u)^{3/2}}+\frac{1}{2c}M\right>\\
\nabla_{a,E,c}P(a,E,c)&=\left<M_c,~T_c,~-\frac{1}{8}\sqrt{\frac{c}{2}}\oint\frac{u^4~du}{R(u)^{3/2}}+\frac{1}{c}P-\frac{3}{4}K\right>
\end{align*}
In particular, we see that each of the above derivatives can be expressed as linear combinations of the
classical action $K$, the $k^{th}$ moment of the solution $u$
\[
\mu_k=\oint\frac{u^k~du}{\sqrt{R(u)}}
\]
and an integral of the form
\[
I_k=\oint\frac{u^k~du}{R(u)^{3/2}}.
\]
We now show that in the case of the mBBM equation, we can find a closed system of seven equations for the integrals $I_k$
in terms of the traveling wave parameters and the quantities $T$, $M$, $P$ and $K$.  To this end, first notice that for any
integer $j\geq 1$ one has the identity
\begin{align*}
\mu_j&=\oint\frac{u^j~du}{\sqrt{R(u)}}=\oint\frac{u^jR(u)du}{R(u)^{3/2}}\\
&=E~I_j+a~I_{j+1}+\left(\frac{c-1}{2}\right)I_{j+2}-\frac{1}{4}I_{j+4}.
\end{align*}
Moreover, integration by parts yields the relation
\begin{align*}
2j\mu_{j-1}&\oint\frac{u^{j-1}~du}{\sqrt{R(u)}}=\oint\frac{u^j R'(u)du}{R(u)^{3/2}}\\
&=aI_j+(c-1)I_{j+1}-I_{j+2}.
\end{align*}
This yields a linear system of seven equations in seven unknowns $\{I_{j}\}_{j=0}^{6}$ known as the Picard-Fuchs system:
\[
\left(
  \begin{array}{ccccccc}
    E & a & \frac{c-1}{2} & 0 & -\frac{1}{4} & 0 & 0\\
    0 & E & a & \frac{c-1}{2} & 0 & -\frac{1}{4} & 0 \\
    0 & 0 & E & a & \frac{c-1}{2} & 0 & -\frac{1}{4}\\
    a & (c-1) & 0 & -1 & 0 & 0 & 0 \\
    0 & a & (c-1) & 0 & -1 & 0 & 0 \\
    0 & 0 & a & (c-1) & 0 & -1 & 0 \\
    0 & 0 & 0 & a & (c-1) & 0 & -1
  \end{array}
\right)
\left(
  \begin{array}{c}
    I_0 \\
    I_1 \\
    I_2 \\
    I_3 \\
    I_4 \\
    I_5\\
    I_6\\
  \end{array}
\right)
=
\left(
  \begin{array}{c}
    \sqrt{\frac{2}{c}}T \\
    \sqrt{\frac{2}{c}}M \\
    \mu_2\\
    0 \\
    2\sqrt{\frac{2}{c}}T \\
    4\sqrt{\frac{2}{c}}M \\
    6\mu_2
  \end{array}
\right).
\]
The matrix which arises in the above linear system is known as the Sylvester matrix of the polynomials
$R(u)$ and $R'(u)$.  In commutative algebra, a standard result is that the determinant of the Sylvester
matrix of two polynomials $P(u)$ and $Q(u)$ vanishes if and only if they have a common root.  If we now
restrict ourselves to $a=0$, it follows that the above matrix is invertible for all periodic traveling wave solutions $u(x;0,E,c)$
of the corresponding mBBM equation.  
In particular, we can solve the above Picard-Fuchs system and use the fact that $M_a=-\frac{1}{2}\sqrt{\frac{c}{2}}I_2$
to see that
\[
\lim_{E\to 0^+}\frac{{\rm disc}(R(u;0,E,c))M_a(0,E,c)}{(c-1)T(0,E,c)}=1
\]
as claimed.

\bibliography{GBBMbib}

\begin{thebibliography}{10}

\bibitem{AP}
J.~Angulo~Pava.
\newblock Nonlinear stability of periodic traveling wave solutions to the
  {S}chr\"odinger and the modified {K}orteweg-de {V}ries equations.
\newblock {\em J. Differential Equations}, 235(1):1--30, 2007.

\bibitem{APBS}
J.~Angulo~Pava, J.~L. Bona, and M.~Scialom.
\newblock Stability of cnoidal waves.
\newblock {\em Adv. Differential Equations}, 11(12):1321--1374, 2006.

\bibitem{Baum}
H.~Baumg\"artel.
\newblock {\em Analytic Perturbation Theory for Matrices and Operators},
  volume~15 of {\em Operator theory: Advances and Applications}.
\newblock Birkh\"auser, 1985.

\bibitem{BBM}
T.~B. Benjamin, J.~L. Bona, and J.~J. Mahony.
\newblock Model equations for long waves in nonlinear dispersive systems.
\newblock {\em Philos. Trans. Roy. Soc. London Ser. A}, 272(1220):47--78, 1972.

\bibitem{BS}
J.~L. Bona and A.~Soyeur.
\newblock On the stability of solitary-waves solutions of model equations for
  long waves.
\newblock {\em J. Nonlinear Sci.}, 4(5):449--470, 1994.

\bibitem{BD}
N.~Bottman and B.~Deconinck.
\newblock Kdv cnoidal waves are linearly stable.
\newblock Submitted, 2009.

\bibitem{B}
T.~J. Bridges.
\newblock Multi-symplectic structures and wave propagation.
\newblock {\em Math. Proc. Cambridge Philos. Soc.}, 121(1):147--190, 1997.

\bibitem{BrJ}
J.~C. Bronski and M.~A. Johnson.
\newblock The modulational instability for a generalized korteweg-de vries
  equation.
\newblock 2008.
\newblock Submitted.

\bibitem{BrJK}
J.~C. Bronski, M.~A. Johnson, and T.~Kapitula.
\newblock An index theorem for the stability of periodic traveling waves of kdv
  type.
\newblock Preprint.

\bibitem{Bridges:Maslov}
F.~Chardard, F.~Dias, and T.~J. Bridges.
\newblock Computing the maslov index of solitary waves. part 1: Hamiltonian
  systems on a 4-dimensional phase space.
\newblock 2008.
\newblock preprint.

\bibitem{DK}
B.~Deconinck and T.~Kapitula.
\newblock On the orbital (in)stability of spatially periodic stationary
  solutions of generalized korteweg-de vries equations.
\newblock 2009.
\newblock Submitted.

\bibitem{Gard}
R.~A. Gardner.
\newblock Spectral analysis of long wavelength periodic waves and applications.
\newblock {\em J. Reine Angew. Math.}, 491:149--181, 1997.

\bibitem{GSS}
M.~Grillakis, J.~Shatah, and W.~Strauss.
\newblock Stability theory of solitary waves in the presence of symmetry.
  {I,II}.
\newblock {\em J. Funct. Anal.}, 74(1):160--197,308--348, 1987.

\bibitem{Gun}
R.~Gunning.
\newblock {\em Lectures on Complex Analytic Varieties: Finite Analytic Mapping
  (Mathematical Notes)}.
\newblock Princeton Univ Press, 1974.

\bibitem{Haragus}
M.~H{\u{a}}r{\u{a}}gu{\c{s}}.
\newblock Stability of periodic waves for the generalized {BBM} equation.
\newblock {\em Rev. Roumaine Math. Pures Appl.}, 53(5-6):445--463, 2008.

\bibitem{HK}
M.~H{\v{a}}r{\v{a}}gu{\c{s}} and T.~Kapitula.
\newblock On the spectra of periodic waves for infinite-dimensional
  {H}amiltonian sytems.
\newblock {\em To appear: Physica D}.

\bibitem{MJ1}
M.~A. Johnson.
\newblock Nonlinear stability of periodic traveling wave solutions of the
  generalized korteweg-de vries equation.
\newblock 2009.
\newblock Submitted.

\bibitem{PW1}
R.~L. Pego and M.~I. Weinstein.
\newblock Eigenvalues, and instabilities of solitary waves.
\newblock {\em Philos. Trans. Roy. Soc. London Ser. A}, 340(1656):47--94, 1992.

\bibitem{DHP}
D.~H. Peregrine.
\newblock Long waves on a beach.
\newblock {\em J. Fluid Mechanics}, 27:815–827, 1967.

\bibitem{MW1}
M.~I. Weinstein.
\newblock Existence and dynamic stability of solitary wave solutions of
  equations arising in long wave propagation.
\newblock {\em Comm. Partial Differential Equations}, 12(10):1133--1173, 1987.

\bibitem{W1}
G.~B. Whitham.
\newblock {\em Linear and nonlinear waves}.
\newblock Pure and Applied Mathematics (New York). John Wiley \& Sons Inc., New
  York, 1999.
\newblock Reprint of the 1974 original, A Wiley-Interscience Publication.

\end{thebibliography}

\end{document}